\documentclass{aic}

\aicAUTHORdetails{%
  title = {There are Only a Finite Number of Excluded Minors for the Class of Bicircular Matroids}, 
  author = {Matt DeVos, Daryl Funk, Luis Goddyn, and Gordon Royle},
  plaintextauthor = {Matt DeVos, Daryl Funk, Luis Goddyn, Gordon Royle},
  runningtitle = {Excluded minors for bicircular matroids}, 
}   

\aicEDITORdetails{%
   year={2023},
   number={7},
   received={22 February 2021},   
   revised={4 November 2022},    
   published={17 November 2023},  
   doi={10.19086/aic.2023.7},      
}   

\usepackage{amsthm}
\usepackage[shortlabels]{enumitem}
\usepackage{mathrsfs}
\usepackage{hyperref} 
\newtheorem{mainthm}{Theorem}
\newtheorem{thm}{Theorem}[section]
\newtheorem{lem}[thm]{Lemma}
\newtheorem{prop}[thm]{Proposition} 
\newtheorem*{claim}{Claim}
\newtheorem{claimnumbered}{Claim}[thm]
\DeclareMathOperator{\Star}{star}  
\DeclareMathOperator{\cl}{cl}  
\DeclareMathOperator{\co}{co}  
\DeclareMathOperator{\si}{si}  
\newcommand{\del}{\backslash}
\newcommand{\inv}{^{-1}}
\newcommand{\sd}{\triangle}
\newcommand{\Bb}{\mathcal B}
\newcommand{\Ii}{\mathcal I}
\newcommand{\Cc}{\mathcal C}
\newcommand{\Ggg}{\mathscr G}

\begin{document}

\begin{frontmatter}[classification=text]


\author[matt]{Matt DeVos}
\author[daryl]{Daryl Funk}
\author[luis]{Luis Goddyn}
\author[gordon]{Gordon Royle}

\begin{abstract}
We show that the class of bicircular matroids has only a finite number of excluded minors. 
Key tools used in our proof include representations of matroids by biased graphs and the recently introduced class of quasi-graphic matroids. 
We show that if $N$ is an excluded minor of rank at least {ten}, then $N$ is quasi-graphic. 
Several small excluded minors are quasi-graphic. 
Using biased-graphic representations, we find that $N$ already contains one of these. 
We also provide an upper bound, in terms of rank, on the number of elements in an excluded minor, so the result follows. 
\end{abstract}
\end{frontmatter}


\section{Introduction}

Let $G=(V,E)$ be a graph, possibly with loops or multiple edges. 
The subgraphs of $G$ that are minimal with respect to having more edges than vertices are the \emph{bicycles} of $G$ 
{(Figure \ref{figbicycles})}. 
\begin{figure}[tp] 
\begin{center} 
\includegraphics[scale=0.95]{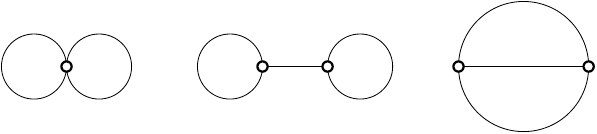}
\end{center} 
\caption{{A bicycle is a subdivision of one of these three graphs.}}
\label{figbicycles} 
\end{figure} 
The \emph{bicircular matroid} of $G$, denoted $B(G)$, is the matroid on ground set $E$ whose circuits are the edge sets of bicycles of $G$. 
A loopless matroid $M$ is \emph{bicircular} if there is a graph $G$ for which $B(G)$ is isomorphic to $M$. 

The class of bicircular matroids is minor-closed (after allowing for loops, for instance by including direct sums with rank-0 uniform matroids). 
{We have found twenty-seven excluded minors for the class of bicircular matroids by search using a computer}
{(see \hyperref[appendix]{Appendix}).}
We conjecture that this list is complete. 
Here we prove that the list is finite: 

\begin{mainthm} \label{main} 
There are only a finite number of excluded minors for the class of bicircular matroids. 
\end{mainthm} 

We prove Theorem \ref{main} by showing that there are no excluded minors of rank greater than
nine (Theorem \ref{Rachel}), 
then bounding the number of elements, in terms of rank, that an excluded minor may have (thus proving Theorem \ref{purple}). 

Bicircular matroids form a natural and fundamental class of matroids. 
{They are precisely the transversal matroids arising from set systems in which each element is in at most two sets.}
When considered within the larger classes containing them,  bicircular matroids appear as natural and fundamental as the class of graphic matroids. 
{Like graphic matroids, they form a natural subclass of the class of frame matroids. 
Chen and Geelen have shown that there are infinitely many excluded minors for the class of frame matroids~\cite{CHEN201846}. 
There are many other natural subclasses of frame matroids that have been extensively considered in the literature; signed-graphic matroids and gain-graphic matroids, for instance. Little is known of excluded minors for these classes.}

Perhaps the most natural class of matroids in which we should situate ourselves is the class of \emph{quasi-graphic} matroids, introduced by Geelen, Gerards, and Whittle~\cite{JGT:JGT22177}. 
For a vertex $v$ in a graph $G$, denote by $\Star(v)$ the set of edges incident to $v$ whose other end is in $V(G)-\{v\}$ (thus $\Star(v)$ contains no loops). 
A \emph{framework} for a matroid $M$ is a graph $G$ 
with $E(G)=E(M)$ such that 
(i) the edge set of each component of $G$ has rank in $M$ no larger than the number of its vertices, 
(ii) for each vertex $v \in V(G)$, $\cl_M(E(G-v)) \subseteq E(G) - \Star(v)$, and 
(iii) no circuit of $M$ induces a subgraph of $G$ with more than two components. 
A matroid is \emph{quasi-graphic} if it has a framework. 

Both the cycle matroid $M(G)$ and the bicircular matroid $B(G)$ of any graph $G$ are quasi-graphic, sharing the framework $G$. 
There will be in general many quasi-graphic matroids sharing $G$ as a framework. 
The cycle matroid and the bicircular matroid occupy two extremes amongst the quasi-graphic matroids, in the following sense. 
Let $M$ be a connected quasi-graphic matroid. 
Then $M$ has a connected framework $G$ (by Corollary 4.7 of~\cite{MR4037634}). 
The cycle matroid $M(G)$ of $G$ is the tightest possible quasi-graphic matroid with framework $G$. 
That is, 
if $X \subseteq E(G)$ is a dependent set in a quasi-graphic matroid with framework $G$, then $X$ is dependent in $M(G)$. 
The bicircular matroid $B(G)$ of $G$ is the freest possible quasi-graphic matroid with framework $G$. 
That is, every independent set in any quasi-graphic matroid with framework $G$ is independent in $B(G)$.  

The class of graphic matroids is minor-closed. 
In 1959 Tutte showed that the class of graphic matroids is characterised by a list of just five excluded minors (see~\cite[Theorem 10.3.1]{oxley}). 
Our strategy for showing that likewise, the bicircular matroids are characterised by a finite list of excluded minors, is as follows. 
Let $N$ be an excluded minor for the class of bicircular matroids. 
We show that $N$ cannot be large, neither in rank 
nor number of elements. 
The more challenging part is in bounding the rank. 
It is not hard to show that $N$ must be vertically 3-connected. 
Under the assumption that $N$ has rank greater than 
{nine}, 
we find a 3-element set $X \subseteq E(N)$ for which $N/Z$ remains vertically 3-connected, for every subset $Z$ of $X$. 
Wagner's characterisation~\cite{MR815399} of the graphs representing a vertically 3-connected bicircular matroid of rank at least five enables us 
to construct a bicircular \emph{twin} $M$ for $N$: that is, a bicircular matroid $M$ with $E(M)=E(N)$ for which $M/Z=N/Z$ for each 
{non-empty} 
subset $Z \subseteq X$. 
We use the twin $M$ to show that $N$ is quasi-graphic, sharing a framework with $M$. 
Using \emph{biased-graphic} representations (described in Section \ref{eintanembau}), it is then relatively straightforward to derive the contradiction that $N$ properly contains an excluded minor for the class of bicircular matroids.

We begin with some elementary facts regarding the structure of bicircular matroids that are not vertically 3-connected, and show that excluded minors are vertically 3-connected (Section \ref{secconsep}). 
In Section \ref{representations} we more precisely describe representations of bicircular matroids by graphs. 
Sections \ref{secfind2} and \ref{findingX} are devoted to constructing a bicircular twin for an excluded minor of rank 
greater than {nine}. 
In Section \ref{wonderwoman}, we use the results of Sections \ref{secfind2} and \ref{findingX} to bound the rank of an excluded minor, at no more than 
{nine}. 
Finally, in Section \ref{greenhornet}, we bound the number of elements in an excluded minor as a function of its rank. 
Together, the main results of Sections \ref{wonderwoman} and \ref{greenhornet} prove Theorem \ref{main}. 

{
A proof that our list of excluded minors is complete will need a new idea. 
We obtained our list by identifying excluded minors in Mayhew and Royle's database of all matroids on up to nine elements~\cite{MR2389607} (see \hyperref[appendix]{Appendix}). 
Combining this exhaustive search with the results of Section \ref{greenhornet}, we now know that any excluded minor not on our list has at least ten elements and rank at most nine. 
Further, we know that when such an excluded minor has rank $r$ it has at most $3{r \choose 2} + r + 4$ elements (Lemma \ref{pen}). 
Unfortunately, this bound is too crude to be of any practical use in either finding an additional excluded minor or ruling one out by exhaustive search (giving an upper bound of 121 elements in a potential excluded minor of rank 9, for instance). 
The gap between nine elements and $3{r \choose 2} + r + 4$ elements with $r \leq 9$ is simply too large. 
}

\section{Connectivity and separations in bicircular matroids} \label{secconsep} 

A \emph{separation} of a matroid $M$ is a partition $(A,B)$ of its ground set; we call $A$ and $B$ the \emph{sides} of the separation. 
A \emph{$k$-separation} of $M$ is a separation $(A,B)$ with both $|A|, |B| \geq k$ and $r(A) + r(B) - r(E) < k$.  
A matroid $M$ is \emph{$k$-connected} if $M$ has no $l$-separation with $l<k$. 
A \emph{vertical $k$-separation} of $M$ is a $k$-separation $(A,B)$ with both $r(A), r(B) \geq k$. 
A matroid is \emph{vertically $k$-connected} if it has no vertical $l$-separation with $l<k$. 
Note that if $(A,B)$ is a 2-separation in a connected matroid that is not vertical, then one of $A$ or $B$ is a subset of a parallel class.

A \emph{separation} of a graph $G$ is a partition $(A,B)$ of its edge set; $A$ and $B$ are its \emph{sides}. 
For a subset $A$ of edges of $G$, we denote by $V(A)$ the set of vertices of $G$ with an incident edge in $A$. 
A \emph{$k$-separation} of $G$ is a separation $(A,B)$ with 
$|V(A) \cap V(B)| = k$. 
A separation $(A,B)$ is \emph{proper} if both $V(A)-V(B)$ and $V(B)-V(A)$ are non-empty. 
A graph is \emph{$k$-connected} if it has at least $k+1$ vertices and has no proper $l$-separation with $l < k$.
A loop in a graph is an edge with just one end and a loop in a matroid is a circuit of size one; context will prevent any possible confusion. 

Let $M$ and $M'$ be matroids on ground sets $E$ and $E'$, respectively. 
Assume $E \cap E' = \{e\}$, and that $e$ is neither a loop nor coloop in either $M$ or $M'$. 
The \emph{2-sum} of $M$ and $M'$ \emph{on} \emph{basepoint} $e$ is the matroid, denoted $M \oplus_2 M'$, on ground set $(E \cup E') - e$ whose circuits are the circuits of $M \del e$ and $M' \del e$ along with all sets of the form $(C-e) \cup (C'-e)$, where $C$ is a circuit of $M$ containing $e$ and $C'$ is a circuit of $M'$ containing $e$. 
A connected matroid $M$ has a 2-separation if and only if it can written as a 2-sum $M = M_1 \oplus_2 M_2$, where each of $M_1$ and $M_2$ has at least three elements and is isomorphic to a proper minor of $M$~\cite[Theorem 8.3.1]{oxley}. 
Thus $M$ has a vertical 2-separation if and only if $M$ can be written as a 2-sum in which each of the summands is a proper minor of $M$ which has at least three elements and rank at least two. 

The class of bicircular matroids is not closed under 2-sums. 
However, there is a special case in which we can find a graph representing a 2-sum of two bicircular matroids. 
Let $G$ and $G'$ be graphs with $E(G) \cap E(G') = \{e\}$, where $e$ is neither a loop nor coloop in $B(G)$ or $B(G')$. 
If $e$ is a loop in $G$ incident with $v \in V(G)$ and $e$ is a loop in $G'$ incident with $v' \in V(G')$, then define the \emph{loop-sum} of $G$ and $G'$ \emph{on} $e$ to be the graph obtained from the disjoint union of $G \del e$ and $G' \del e$ by identifying vertices $v$ and $v'$. 

\begin{prop} \label{clarinet}
Let $G$ and $G'$ be graphs with $E(G) \cap E(G') = \{e\}$. 
If $H$ is the loop-sum of $G$ and $G'$ on $e$, 
then $B(H)$ is the 2-sum of $B(G)$ and $B(G')$ on $e$.
\end{prop}

\begin{proof}
It is straightforward to check that the circuits of $B(H)$ and the circuits of the two sum $B(G) \oplus_2 B(G')$ coincide.  
\end{proof} 

Equipped with this basic tool, we can now show that an excluded minor for the class of bicircular matroids is vertically 3-connected. 
We also use the notion of the parallel connection of two matroids $M, M'$ on an element $e$, denoted $P(M,M')$. 
The 2-sum of $M$ and $M'$ on basepoint $e$ is given by deleting $e$ from the parallel connection of $M$ and $M'$ on basepoint $e$: $M \oplus_2 M' = P(M,M') \del e$. 
(See~\cite[Section 7.1]{oxley} for details). 

\begin{lem} \label{Nvert3} 
\addcontentsline{toc}{subsubsection}{$N$ is vertically 3-conn}
Let $N$ be an excluded minor for the class of bicircular matroids. 
Then $N$ is vertically 3-connected, and every parallel class has size at most $2$.
\end{lem}

\begin{proof}
Suppose $N$ contains three elements $e, f, g$ in a common parallel class. 
Let $M = N \del g$, and let $G$ be a graph representing $M$. 
The circuit $\{e,f\}$ forms a bicycle in $G$, so $e$ and $f$ are a pair of loops incident to a common vertex $v$. 
Let $H$ be the graph obtained from $G$ by adding $g$ as a loop incident to $v$: $H$ is a bicircular representation of $N$, a contradiction. 

Now suppose $N$ contains a series pair $e, f$. 
Let $M = N/f$, and let $G$ be a graph representing $M$. 
Let $H$ be the graph obtained from $G$ by subdividing $e$, and label the two edges of the resulting path $e, f$. 
Then $H$ is a bicircular representation of $N$, a contradiction. 

Finally, suppose $N$ has a vertical 2-separation, where neither side is a series class. 
Then $N$ is a 2-sum $N = M_1 \oplus_2 M_2$, where each of $M_1, M_2$ has at least three elements and rank at least two, neither is a circuit, and each is isomorphic to a proper minor of $N$. 
Each of $M_1$ and $M_2$ is therefore bicircular. 
We construct $M_1$ and $M_2$ as minors of $N$ as follows. 
As $M_2$ has rank at least two and is not a circuit, and $e$ is neither a loop nor coloop in $M_2$, we may choose a circuit $C$ of $M_2$ containing $e$ as well as a second element $e'$. 
Let $M_1^+$ be the matroid obtained from the parallel connection $P(M_1,M_2)$ of $M_1$ and $M_2$ on $e$ by deleting all elements of $E(M_2)$ but $C$ and contracting all elements of $C$ but $e$ and $e'$. 
Elements $e$ and $e'$ are a parallel pair in $M_1^+$, and $M_1 = M_1^+ \del e'$.  
Let $G_1^+$ be a graph representing $M_1^+$. 
As $\{e,e'\}$ is a parallel pair of $M_1^+$, $e$ and $e'$ are a pair of loops incident to a common vertex of $G_1^+$. 
Let $G_1 = G_1^+ \del e'$. 
Then $G_1$ is a representation for $M_1$, and $e$ is a loop in $G_1$. 
Similarly, we may construct a graph $G_2$ representing $M_2$ in which $e$ is a loop. 
Thus by Proposition \ref{clarinet} the graph $H$ obtained as the loop-sum of $G_1$ and $G_2$ on $e$ is a bicircular representation for $N$, a contradiction. 
\end{proof} 

We say a matroid $M$ is \emph{essentially 3-connected} if it is connected and 
its cosimplification is vertically 3-connected.  
Equivalently, a connected matroid is essentially 3-connected if 
every 2-separation $(A,B)$ has the property that one of $A$ or $B$ is a collection of 
pairwise disjoint sets, each of which is contained in a 
series class, which after cosimplification is contained in a parallel class (possibly trivial). 
If $(A,B)$ is 2-separation of $M$ and the elements of $B$ remaining in a cosimplification $\co(M)$ of $M$ are contained in a parallel class of $\co(M)$, then we say $B$ \emph{cosimplifies to a parallel class}. 
Accordingly, we say that a 2-separation $(A,B)$ is \emph{essential} if neither $A$ nor $B$ 
cosimplifies to a parallel class. 
Thus a connected matroid is {essentially 3-connected} precisely when it has no essential 2-separation. 
Note that the definitions say that a matroid is vertically 3-connected if and only if it is essentially 3-connected and has no non-trivial series class. 

Let $M$ be a bicircular matroid represented by the graph $G=(V,E)$. 
For a subset $X \subseteq E$, denote by 
$a(X)$ the number of acyclic components of the subgraph $G[X]$ induced by the edges in $X$. 
Then the rank of a set $X \subseteq E$ is $r(X) = |V(X)| - a(X)$. 
Given a separation $(A,B)$ of the matroid, we call the {connected} 
components of $G[A]$ and $G[B]$ the \emph{parts} of the separation in $G$. 
A \emph{balloon} in $G$ is a maximal set of edges $X$ such that $G[X]$ is a subdivision of either a loop or the graph consisting of a single edge with a loop incident to one of its ends, and $V(X) \cap V(E(G)-X) = \{v\}$ for some vertex $v$, where if $G[X]$ has a vertex of degree 1, then this vertex is $v$; $v$ is the \emph{vertex of attachment} of the balloon. 
{The following lemma is a straightforward consequence of the proof of Lemma 2.5 in~\cite{MR3588716}; we provide a proof specialised to the bicircular case.}

\begin{lem} \label{bobby} 
Let $M$ be a connected bicircular matroid represented by the graph $G$. 
A partition $(A,B)$ of $E(G)$ is an essential 2-separation of $M$ if and only if 
there are just two parts of $(A,B)$ that contain 
{cycles}, 
one $A'$ contained in $A$ and the other $B'$ contained in $B$, $G$ consists of $A'$ and $B'$ together with a path between them, and neither 
$A'$ nor $B'$ consists of just a collection of balloons all sharing a common vertex of attachment.
\end{lem}

\begin{proof} 
Let $(A,B)$ be an essential 2-separation of $M$. 
Since $M$ is connected, $G$ is connected. 
Write $S = V(A) \cap V(B)$. For each part $W$, let $\delta(W)=1$ if $W$ is acyclic and $\delta(W)=0$ otherwise. Let $A_1, \ldots, A_k$ be the parts of $(A,B)$ contained in $G[A]$, and let $B_1, \ldots, B_l$ be the parts of $(A,B)$ contained in $G[B]$. The rank calculation $r(A)+r(B)-r(M) = |V(A)|-a(A)+|V(B)|-a(B)-|V(G)| = 1$ implies 
$|S| - \sum_{i=1}^h \delta(A_i) - \sum_{j=1}^l \delta(B_j) = 1$. 
Since each vertex in $S$ is in exactly one part contained in $G[A]$ and exactly one part contained in $G[B]$, doubling and rearranging this equation yields the sum 
\[ \sum_{i=1}^k \left( |S \cap V(A_i)| - 2\delta(A_i) \right) + \sum_{j=1}^l \left( |S \cap V(B_j)| - 2\delta(B_j) \right) = 2 \] 
Every acyclic part must have at least two vertices in $S$, else $M$ is not connected (every edge in such a part would be a coloop of $M$). Thus every term in the sum is nonnegative. In particular, letting $t$ be the number of vertices in $S$ contained in a part, an acyclic part will contribute $t-2$ to the sum, while each part containing a cycle will contribute $t$. Call a part \emph{neutral} if it is acyclic and contains exactly two vertices in $S$; call a part \emph{cyclic} if it contains a cycle. Since the sum is two, the possibilities for the parts of $(A,B)$ are: 
\begin{enumerate}[label={(\roman*)}]
\item two cyclic parts each with one vertex in $S$ and all other parts neutral, 
\item one cyclic part with two vertices in $S$ and all other parts neutral, 
\item one acyclic part with three vertices in $S$, one cyclic part with one vertex in $S$, and all other parts neutral, 
\item two acyclic parts with three vertices in $S$ and all other parts neutral, or 
\item one acyclic part with four vertices in $S$ and all other parts neutral. 
\end{enumerate} 
Let $W$ be an acyclic part. Since $G[W]$ is connected and acyclic, $G[W]$ is a tree. If $G[W]$ had a leaf $w \notin S$, then the edge incident to $w$ would be a coloop of $M$, contrary to the fact that $M$ is connected. Thus the leaves of each acyclic part are all contained in $S$. This implies that each neutral part is a path linking a pair of vertices in $S$. 
Possibilities (i)-(v) are illustrated in Figure \ref{dadadada}. 
\begin{figure}[tbp] 
\begin{center} 
\includegraphics[scale=0.95]{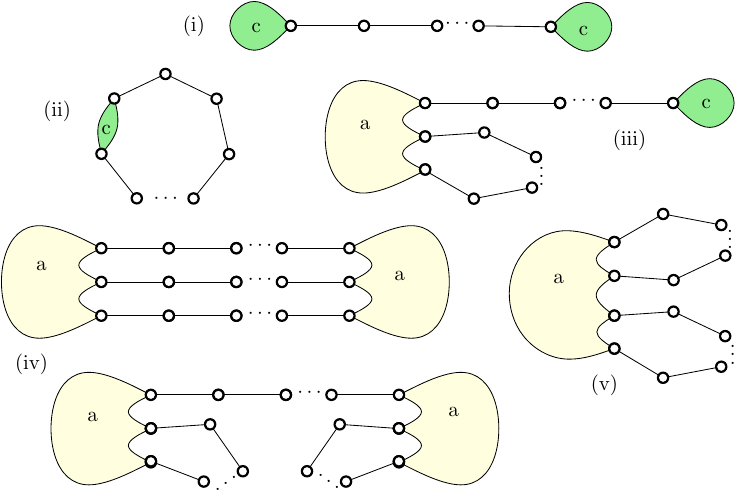}
\end{center} 
\caption{{The possibilities for parts of a 2-separation $(A,B)$. Parts labelled ``c" are cyclic, parts labelled ``a'' are acyclic, and neutral parts are depicted as line segments. Vertices in $V(A) \cap V(B)$ are shown as open circles.}}
\label{dadadada} 
\end{figure} 
Each of possibilities (ii)-(v) has one of $A$ or $B$ as a collection of elements all in series, which therefore cosimplifies to a (trivial) parallel class. Since $(A,B)$ is essential, we have outcome (i): two cyclic parts each with one vertex in $S$ and all other parts neutral. 
Note that this implies the union of the neutral parts is a path linking the two cyclic parts. 
{If both cyclic parts} were contained in one side of the separation, say $A$, then $B$ would have all of its elements contained in a common series class of $M$, 
{again} contradicting the fact that $(A,B)$ is essential. 
A collection of balloons all sharing a common vertex of attachment cosimplifies to a parallel class. 
Thus neither $A'$ nor $B'$ 
consists of just a collection of balloons all sharing a common vertex of attachment. 

Conversely, suppose $(A,B)$ is a partition of $E(M)$ that has the form described in the statement of the lemma. 
Then neither
 $A$ nor $B$ cosimplifies to a parallel class, 
and the connectivity calculation for the separation yields 
$r(A) + r(B) - r(M) = |V(A) \cap V(B)| - a(A) - a(B) = 1$
so $(A,B)$ is an essential 2-separation of $M$. 
\end{proof}

We now describe how a bicircular matroid with an essential 2-separation may become essentially 3-connected after deleting or contracting an element. 


Recall that a \emph{balloon} in a bicircular representation $G$ for a matroid is a maximal set of edges $X$ such that $G[X]$ is a subdivision of either a loop, or the graph consisting of a single edge with a loop incident to one of its ends, and $V(X) \cap V(E(G)-X) = \{v\}$, where if $G[X]$ has a vertex of degree 1, then this vertex is $v$; we call $v$ the balloon's \emph{vertex of attachment}. 
The vertices in $V(X)-\{v\}$ are the \emph{internal vertices} of the balloon. 
A \emph{line} of $G$ is a set of edges $Y$ not contained in a balloon such that $G[Y]$ is a maximal path all of whose internal vertices have degree 2 in $G$ and whose end vertices have degree at least 3. 
The end vertices of a subgraph induced by a line are the \emph{vertices of attachment} of the line; internal vertices of the path are the \emph{internal vertices} of the line. 
{While balloons and lines are sets of edges, we may speak of a balloon or line when referring to a subgraph induced by a balloon or line, when the meaning is clear from context.}
A balloon or a line is \emph{trivial} if it consists of just a single edge. 
For simplicity, we will 
{most often} 
refer to a trivial balloon as a loop, and a trivial line as an edge. 
Throughout the remainder of this paper, the term \emph{balloon} means ``non-trivial balloon" and the term \emph{line} means ``non-trivial line" 
{unless explicitly stated otherwise}. 
We say a graph $G$ has a balloon \emph{at} $x$ when $G$ has a balloon whose vertex of attachment is $x$. 

The elements in a balloon or a line in a graph $G$ are clearly all contained in a common series class of $B(G)$. 
{As long as $B(G)$ is connected and not just a circuit, the} converse also holds. 

\begin{prop} \label{elise} 
Let $S$ be a subset of elements of a bicircular matroid $M$, 
where $M$ is not a circuit and $M$ is connected aside from loops. 
Let $G$ be a graph representing $M$. 
Then $S$ is a non-trivial series class of $M$ if and only if $S$ induces a balloon or a line of $G$. 
\end{prop}

\begin{proof} 
This follows easily from the fact that a set of edges of $G$ is a circuit of $M$ if and only if it induces a bicycle in $G$. 
\end{proof}

An edge $e$ in a connected graph is a \emph{pendant} if after deleting 
balloons 
it has an end vertex of degree 1 
(so a pendant edge cannot be contained in a balloon). 
Its vertex of degree 1 (after deleting balloons) is its \emph{pendant vertex}. 
A set of lines (any of which may be trivial) is a \emph{pendant set} if they share the same vertices of attachment and after deleting the edges of all lines but one, and replacing this line with a single edge, this remaining edge is a pendant. 

\begin{prop} \label{cam}
\addcontentsline{toc}{subsubsection}{deleting or contracting removes an essential 2-separation}
Let $M$ be a connected matroid with bicircular representation $G$, and suppose $M$ has an essential 2-separation $(A,B)$. 
\begin{enumerate}[label={\upshape (\roman*)}]
\item If there is an element $e \in A$ such that $M/e$ is essentially 3-connected, then 
$G[A]$ consists of a pendant set of 
lines (possibly of size 1) and a (possibly empty) set of 
balloons incident to its pendant vertex, 
and $e$ is the single edge in a trivial line of the pendant set. 
\item If there is an element $e \in A$ such that $M \del e$ is essentially 3-connected, then $G[A] \del e$ is a collection of balloons sharing a common vertex of attachment. 
\end{enumerate}
\end{prop}

\begin{proof} 
The statement describes the possibilities given the structure of $G$ described in Lemma \ref{bobby} under the assumptions that $M/e$ or $M \del e$ is essentially 3-connected. 
\end{proof}

We will have occasion to use the following simple fact more than once. 

\begin{prop} \label{xavier}
\addcontentsline{toc}{subsubsection}{2-connected graph of min deg 3 have deletable edges}
Let $G$ be a 2-connected graph 
of minimum degree 3. 
Let $(A,B)$ be a proper 2-separation of $G$ with $B$ minimal. 
Then there is an edge $e \in B$ such that $G \del e$ remains 2-connected. 
\end{prop}

\begin{proof}
Let $e \in B$. 
If $G \del e$ is 2-connected, we are done. 
If not, then 
{$G \del e$ has a cut vertex $x$ and a proper 1-separation $(C,D)$ with $V(C) \cap V(D) = \{x\}$ and $D \subseteq B \del e$. This implies that either $(C \cup \{e\}, D)$ is a proper 2-separation of $G$, where $V(C \cup \{e\}) \cap V(D) = \{x,y\}$ where $y$ is one of the endpoints of $e$, or that $V(D) - V(C \cup \{e\})$ is empty. By the minimality of $B$ the first outcome does not occur. 
Thus $C = A \cup \{e\}$, $D = B - \{e\}$,}
$|V(B) - V(A)| = 1$, 
{and}
$B$ has the form of one of the graphs shown in Figure \ref{manny1}. 
\begin{figure}[tbp] 
\begin{center} 
\includegraphics[scale=0.95]{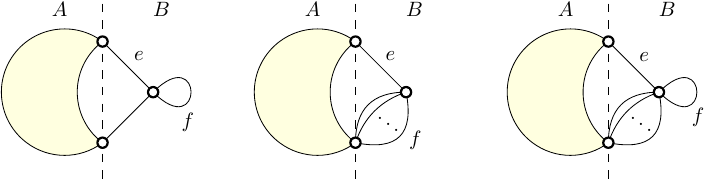}
\end{center} 
\caption{Proper 2-separations $(A,B)$ with $B$ minimal, {where $G \del e$ has a proper 1-separation $(C,D)$ with $D \subset B$, but $(C \cup \{e\}, D)$ is not a proper 2-separation of $G$ because $V(C \cup \{e\}) - V(D)$ is empty}.}
\label{manny1} 
\end{figure} 
In any case, $B$ contains an edge $f$ such that $G \del f$ is 2-connected. 
\end{proof}

\goodbreak
\section{Representations} \label{representations}

\subsection{Frame matroids and biased graphs} \label{eintanembau} 

The class of \emph{frame matroids} is an important subclass of quasi-graphic matroids. 
Cycle matroids of graphs and bicircular matroids are both subclasses of the class of frame matroids. 
Both have natural representations as frame matroids, which we now describe. 

A \emph{framed matroid} is a matroid $M$ having a distinguished basis $F$ with the property that every element of $M$ is either parallel to an element of $F$ or spanned by a pair of elements in $F$. 
Such a basis is a \emph{frame} for $M$. 
A \emph{frame matroid} is a restriction of a framed matroid. 
Given a graph $G = (V,E)$, the set $V$ of vertices of $G$ may naturally be considered to be the frame for a framed matroid on ground set $V \cup E$ in which every edge is spanned by its ends. 
The cycle matroid $M(G)$ of $G$ is the restriction to $E$ of this framed matroid when every cycle of $G$ is a circuit. 
The bicircular matroid $B(G)$ of $G$ is the restriction to $E$ of this framed matroid when every cycle of $G$ is independent. 

This follows from work of Zaslavsky, who showed~\cite{MR1273951} that the class of frame matroids consists precisely of those matroids having a representation as a \emph{biased graph}; that is, a graph $G$ together with a distinguished collection of its cycles $\Bb$ with the property that whenever the intersection of a pair of cycles 
in $\Bb$ is a non-trivial path, then the cycle obtained as their symmetric difference 
is also in $\Bb$. 
This property of $\Bb$ is referred to as the \emph{theta property}; a graph formed by the union of two cycles whose intersection is a non-trivial path is called a \emph{theta}. 
Cycles in $\Bb$ are said to be \emph{balanced}, and are otherwise \emph{unbalanced}. 
A biased graph $(G,\Bb)$ \emph{represents} a frame matroid $M$ if $E(G) = E(M)$ and every circuit of $M$ is either a balanced cycle or a bicycle of $G$ that contains no balanced cycle. 
In this case we write $M = F(G,\Bb)$. 

Let $G$ be a graph. 
Denote by $\Ii(M)$ the collection of independent sets of the matroid $M$. 
The fact that $\Ii(M(G)) \subseteq \Ii(F(G,\Bb)) \subseteq \Ii(B(G))$ holds for any collection $\Bb$ of cycles of $G$ satisfying the theta property is easily seen in terms of biased graph representations. 
It follows directly from the definitions that the cycle matroid of $G$ is the frame matroid $F(G,\Cc(G))$, where $\Cc(G)$ is the set of all cycles of $G$. 
That is, a graphic matroid is a frame matroid having a representation in which all cycles are balanced. 
The bicircular matroid $B(G)$ of $G$ is
{the}
frame matroid $F(G,\emptyset)$. 
That is, $B(G)$ is a frame matroid having a representation in which no cycle is balanced. 

Equipped with these notions, we may now define the class of bicircular matroids in terms of biased graph representations, in such a way that the class is minor-closed and so that minor operations in a graph representation correspond precisely to minor operations in the matroid. 
To that end, let $M$ be a matroid on $E$, let $G=(V,E)$ be a graph, and let $\Bb$ be a subset of the set of loops of $G$. 
We say the matroid $M$ is \emph{bicircular}, \emph{represented} by the biased graph $(G,\Bb)$, if $M = F(G,\Bb)$. 
That is, $M$ is the matroid on $E$ whose collection of circuits consists of the balanced loops of $G$ along with the edge sets of bicycles of $G$ that do not contain a balanced loop. 
In other words, a bicircular matroid is a frame matroid that has a representation $(G,\Bb)$ in which $\Bb$ contains only loops. 

When $(G,\Bb)$ is a biased graph representing a matroid $M$, we call a subgraph induced by the edges of a circuit of $M$ a \emph{circuit-subgraph} of $G$. 
When $M$ is frame and $\Bb$ is empty, the circuit-subgraphs of $G$ are precisely its bicycles. 
It will be convenient to have names for the different forms of circuit-subgraphs. 
We call a pair of unbalanced cycles that meet in exactly one vertex \emph{tight handcuffs}, and 
a pair of disjoint unbalanced cycles together with a minimal path linking them \emph{loose handcuffs}. 
Thus in general, the circuit-subgraphs of a biased graph $(G,\Bb)$ representing a frame matroid are 
balanced cycles, along with 
thetas and handcuffs that contain no balanced cycle. 

It is often convenient to speak of paths, cycles, and subgraphs induced by a set of edges as if they consist of just their edge sets. 
For example, when $M$ is a bicircular matroid represented by the graph $G$, we may say that a theta subgraph $X$ of $G$ is a circuit of $M$. 
By this we mean that $E(X)$ is a circuit of $M$.

\subsection{Minors of bicircular matroids} \label{minorsofbicircularmatroids}

Let $(G,\Bb)$ be a biased graph, where $\Bb$ contains only loops. 
We now define minor operations for $(G,\Bb)$ so that they agree with the corresponding minor operations in the bicircular matroid $F(G,\Bb)$. 
These are specialisations of minor operations for general biased graphs, which are defined so that they agree with their corresponding operations in a represented frame matroid (see~\cite[Sec.\ 6.10]{oxley}). 
These minor operations for biased graphs may also be applied in a framework for a general quasi-graphic matroid so that they agree with the matroid minor, and we will have occasion to do so in Section \ref{warthog}. 
Though extending these operations to the quasi-graphic setting is straightforward, 
some technical details are required to explain them fully, none of which apply for our applications (see~\cite{MR4037634} or~\cite{JGT:JGT22177} for these details). 
We believe the reader will have no trouble finding the required minors. 

To delete an element $e$ from $(G,\Bb)$: 
when $e$ is not a balanced loop, define $(G,\Bb) \del e = (G \del e, \Bb)$, while if $e \in \Bb$ define $(G, \Bb) \del e = (G \del e, \Bb - \{e\})$. 
To contract an element $e$ of $(G,\Bb)$ there are three cases to consider. 
(1) In the case $e$ is a balanced loop, define $(G, \Bb) /e = (G, \Bb) \del e$. 
(2) When $e$ is an edge with distinct ends, 
define $(G,\Bb)/e = (G/e, \Bb)$. 
(3) If $e$ is an unbalanced loop, say incident with the vertex $u$, then let $G^{\circ e}$ be the graph obtained from $G$ by deleting $e$ and replacing each $u$-$v$ edge (where $v \neq u$) with a loop incident with $v$; if $u$ is now isolated, delete $u$, and let $\Bb^{\circ e}$ be the set of loops $\Bb \cup \{e' : e'$ is a loop incident with $u$ in $G\}$. 
Define $(G,\Bb)/e = (G^{\circ e}, \Bb^{\circ e})$. 

The following fact follows immediately from the definitions. 

\begin{prop} \label{wakawaka} 
Let $M$ be a bicircular matroid represented by the biased graph $(G,\Bb)$, where $\Bb$ contains only loops. 
Then for every element $e$ of $M$, 
\begin{itemize} 
\item $M \del e = F((G, \Bb) \del e)$, and 
\item $M/e = F((G, \Bb)/e)$. 
\end{itemize} 
\end{prop}

Since in every case $(G,\Bb) \del e$ and $(G,\Bb)/e$ is a biased graph with no balanced cycles other than loops, it is clear that these definitions yield a minor closed class of frame matroids consisting precisely of the set $\{B(G) : G$ is a graph$\}$ along with all minors of a matroid in this set. 

When $M$ is a loopless bicircular matroid, there is no reason to use the biased graph representation $(G,\emptyset)$, so in this case we may simply refer to the graph $G$ as a \emph{bicircular representation}, or simply a \emph{representation}, for $M$, and write $M = B(G)$.

\subsection{Characterising representations for a bicircular matroid} 

The main difficulty in proving excluded minor theorems for minor-closed classes of matroids lies in the fact that there often may be many different representations for the same matroid. 
The class of bicircular matroids is no different in this respect. 
However, the variability in graph representations for a given bicircular matroid is limited and fully understood. 

Wagner~\cite{MR815399} and Coullard, del Greco, \& Wagner~\cite{COULLARD1991223} characterised the graphs representing a given connected bicircular matroid in terms of the following three operations. 
Let $M$ be a connected bicircular matroid represented by the graph $G$. 
\begin{enumerate}[label=(\arabic*)]
\item \label{type1description} Suppose $G$ has a vertex $v$ such that 
{either $G-v$ is acyclic, or} 
$G$ has a 1-separation $(H,K)$ with $V(H) \cap V(K) = \{v\}$ and 
$H-v$ is acyclic. 
Let $L$ be a
{(possibly trivial)}
line of 
{$G$, contained in} 
$H$ 
{if $G$ has the 1-separation $(H,K)$ at $v$,} 
with end vertices $v$ and $z$. 
Let $e$ be the edge in $L$ incident to $v$, and let $w \in V(L)-v$. 
Let $G'$ be the graph obtained from $G$ by redefining the incidence of $e$ so that $e$ is incident to $w$ instead of $v$. 
Then $G'$ is obtained from $G$ by \emph{rolling $L$ away from $v$}. 
The inverse operation is that of \emph{unrolling $L$ to $v$}. 
Each of $G$ and $G'$ are said to be obtained from the other by a \emph{rolling} operation. 
We call the vertex $v$ an \emph{apex} vertex of $G$. 
\item \label{type2description} 
Suppose $G$ has a vertex $v$ and a vertex $u \neq v$ such that either $G-\{u,v\}$ is acyclic, or $G$ has a 1-separation $(H,K)$ such that $V(H) \cap V(K) = \{v\}$, $u \in V(H) - V(K)$, and $H - \{u,v\}$ is acyclic. 
Assume further that in either case $u$ has degree three and is an end vertex of exactly three 
{(possibly trivial)} 
lines $L_1, L_2, L_3$, where $L_1$ has second end $v$, and 
$L_2$ and $L_3$ share their other end $w \notin \{u,v\}$ in common
{(Figure \ref{rotation_operation})}. 
Let $e_1$ be the edge of $L_1$ incident to $v$, and let $e_2$ be the edge of $L_2$ incident to $w$. 
Let $G'$ be the graph obtained from $G$ by redefining the incidence of $e_1$ and $e_2$ so that $e_1$ is incident to $w$ instead of $v$, and $e_2$ is incident to $v$ instead of $w$. 
Then $G'$ is obtained from $G$ by a \emph{rotation} operation. 
We call $u$ the \emph{rotation vertex} of $G$ and $G'$, and call $L_1$, $L_2$, and $L_3$ the \emph{rotation lines} of $G$; a 
{trivial rotation line} 
is a \emph{rotation edge}. 
We call the vertex $v$ an \emph{apex} vertex of $G$. 
{In general any subset of the rotation lines $L_1$, $L_2$, $L_3$ may consist of trivial lines; in the case that $B(G)$ is cosimple, each of $L_1$, $L_2$, and $L_3$ is a rotation edge. A rotation operation is illustrated in Figure \ref{rotation_operation}.  
\begin{figure}[tbp] 
\begin{center} 
\includegraphics[scale=0.95]{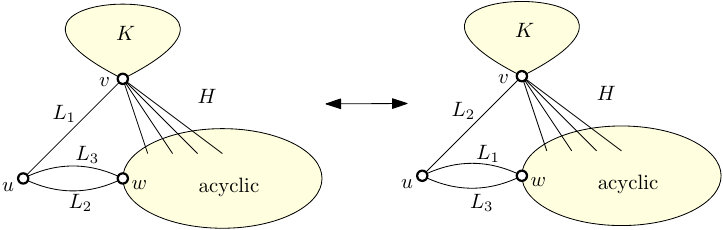}
\end{center} 
\caption{{A rotation operation.}}
\label{rotation_operation} 
\end{figure} }
\item Let $Y$ be a line (respectively, balloon) of $G$. 
Let $G'$ be the graph obtained by replacing 
{$G[Y]$ with a subgraph with edge set $Y$ so that $G'[Y]$ is a line (resp.\ balloon) with the}
same vertices (resp.\ vertex) of attachment. 
Then $G'$ is obtained from $G$ by a \emph{replacement} operation. 
\end{enumerate}

We state the main results of~\cite{COULLARD1991223} and~\cite{MR815399} here slightly more precisely than in the original papers, but in each case the full result stated here is proved in the respective paper, with the exception of the adjective ``vertical" in Theorem \ref{duck} (which is a straightforward extension).  
Theorem \ref{duck} is a representation theorem for 3-connected  bicircular matroids of rank at least 5, while Theorem \ref{goose} does the job for all connected bicircular matroids. 

\begin{thm}[\cite{MR815399}, Theorem 5] \label{duck} 
Let $M$ be a 
vertically 3-connected
matroid of rank at least 5, and let $G$ and $G'$ be bicircular representations for $M$. 
Then $G$ is obtained from $G'$ by either a sequence of rolling operations or a rotation operation. 
\end{thm}


\begin{figure}[tbp] 
\begin{center} 
\includegraphics[scale=0.95]{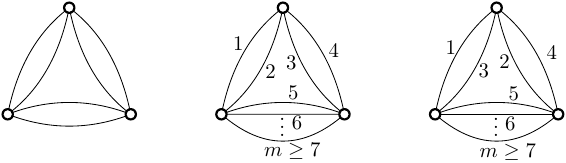} 
\end{center} 
\caption{Rank-3 rearrangements.}
\label{quartetopus1} 
\end{figure} 
\begin{figure}[tbp] 
\begin{center} 
\includegraphics[scale=0.95]{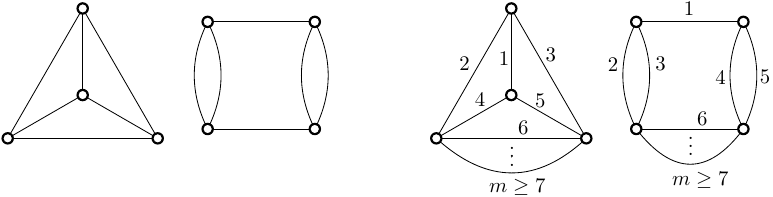}
\end{center} 
\caption{{Rank-4 rearrangments.}}
\label{quartetopus2} 
\end{figure} 
The graphs illustrated in Figures \ref{quartetopus1} and \ref{quartetopus2} show that the assumption of rank at least five in Theorem \ref{duck} is required~\cite{COULLARD1991223}. 
The rank-3 uniform matroid on six elements is represented by the 3-vertex graph with six edges shown in Figure \ref{quartetopus1} with any labelling of its edges. 
Labels may be swapped on edges as indicated in the pair of 3-vertex graphs with $m \geq 7$ edges as shown in Figure \ref{quartetopus1}; both graphs represent the same bicircular matroid. 
The rank-4 uniform matroid on six elements is represented by each of the 4-vertex graphs with six edges shown in Figure \ref{quartetopus2}, with any labelling of their edges. 
For any integer $m \geq 7$, each of the graphs at right in Figure \ref{quartetopus2} has $m-5$ edges linking the pair of vertices labelled $6, \ldots, m$; both graphs provide bicircular representations for the same matroid. 
{Given such a pair of graphs representing the same matroid, we say that one has been obtained from the other as a \emph{rearrangement}, or by a \emph{rearrangement} operation. 
}

For any connected graph $G$, let $\co(G)$ be the graph obtained from $G$ by contracting all but one edge of each balloon and all but one edge of each line of $G$. 
{By Proposition \ref{elise}}, $\co(G)$ provides a representation for $\co(B(G))$. 
Let $\mathscr G$ be the family of graphs $\{ G : \co(G)$ is one of the graphs show in Figure \ref{quartetopus1} or \ref{quartetopus2}$\}$. 

\begin{thm}[\cite{COULLARD1991223}, Theorem 4.11] \label{goose} 
Let $M$ be a connected matroid, and let $G$ and $G'$ be bicircular representations for $M$. 
If 
{not both $G$ and $G'$ are in $\mathscr{G}$,} 
then $G$ is obtained from $G'$ by either a sequence rolling and replacement operations, or a sequence of rotation and replacement operations. 
\end{thm}

{
It is a straightforward consequence of the results in~\cite{COULLARD1991223} that if $B(G)=B(G')$ and $G$ is in $\mathscr G$, then $G'$ is in $\mathscr G$, and, though they may have different vertices of attachment, each line of $G'$ is obtained from a line of $G$ by a replacement operation. 
If $G$ and $G'$ are in $\Ggg$ and $\co(G')$ is a rearrangement of $\co(G)$, then we also say that $G'$ is a \emph{rearrangement} of $G$, and that $G'$ is obtained from $G$ via rearrangement and replacement operations.}
For convenience, we restate Theorem 3.3, 
{along with this slight extension,} 
in the form we will frequently apply it: 

\begin{lem} \label{duckduck} 
Let $M$ be an essentially 3-connected matroid, and let $G$ and $G'$ be bicircular representations for $M$. 
Then $G$ is obtained from $G'$ by either a sequence of rolling and replacement operations, 
a sequence of rotation and replacement operations, 
{or 
a sequence of rearrangement and replacement operations.}
\end{lem}

We say a connected bicircular matroid $M$ is 
\emph{type 1} if it has a bicircular representation $G$
{as described in paragraph \ref{type1description} on page \pageref{type1description}. That is, $G$}
{has} a vertex $v$ such that $G-v$ is acyclic, or $G$ has a 1-separation $(H,K)$ with $V(H) \cap V(K) = \{v\}$ where $H-v$ is acyclic. 
{We say $M$ is}
\emph{type 2} if $M$ is not type 1 but has a representation as described in paragraph \ref{type2description}
{on page \pageref{type2description}}, 
{to which a rotation operation may be applied.}
{We say $M$ is} 
\emph{type 3} otherwise. 
By Theorem \ref{goose}, 
if $M$ has a type 3 representation then $M$ is not type 1 or 2. 
Accordingly, we call a representation witnessing that $M$ is type 1, 2, or 3, respectively, a \emph{type 1, 2, or 3 representation}.  
By 
{Theorem \ref{goose},} 
if 
{$\co(M)$} 
has rank at least five, {and $M$} is essentially 3-connected and type 3, then $M$ has a unique bicircular representation up to replacement. 
Theorem \ref{duck} says that if $M$ has rank at least five, is type 3 and vertically 3-connected, then $M$ has a unique graph representation. 

An apex vertex is almost always unique, as we note in the following straightforward proposition. 

\begin{prop} \label{serabande} 
Let $G$ be a type 1 or type 2 representation of a connected bicircular matroid. 
Then $G$ has a unique apex vertex unless $\co(G)$ is one of the graphs shown in Figure \ref{nonuniqueapex}, in which case $v_1$ and $v_2$ are both apex vertices of $G$. 
\end{prop}

\begin{figure}[tbp] 
\begin{center} 
\includegraphics[scale=0.95]{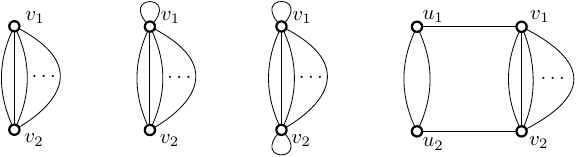}
\end{center} 
\caption{Graphs with more than one apex; there may be any number of $v_1$-$v_2$ edges. }
\label{nonuniqueapex} 
\end{figure} 

An immediate consequence of Proposition \ref{serabande} is the following. 

\begin{prop} \label{wawawawa}
Let $G$ be a type 1 or type 2 representation of a vertically 3-connected matroid of rank at least 5. 
Then $G$ has a unique apex vertex. 
\end{prop}

\section{Finding 2-connected type 3 representations} \label{secfind2}

Let $M$ and $N$ be matroids on a common ground set $E$.  
We say that $M$ and $N$ are \emph{twins relative to} a subset $X$ of $E$, or that \emph{$M$ is a twin for $N$ relative to $X$}, if $M /e = N/e$ for every $e \in X$.  

Let $N$ be an excluded minor for the class of bicircular matroids, and assume $N$ has rank at least 
{ten}. 
Our proof strategy is to construct a bicircular twin $M$ for $N$ relative to a subset $X$ of their common ground set $E$. 
For reasons that will become clear in the course of the proof, 
{the}
set $X$ will contain three elements, and have the following properties: 
\begin{itemize} 
\item  for every subset $Z \subseteq X$, the matroid $N/Z = M/Z$ is vertically 3-connected, and 
\item in a graph representing $M$, the subgraph induced by $X$ has at least five vertices. 
\end{itemize} 
We then use the bicircular twin $M$ to deduce the contradiction that $N$ already contains a smaller excluded minor. 
To do so, our first step is to find an element $e$ for which $N \del e$ is essentially 3-connected having a 2-connected bicircular representation. 

%

\subsection{Series classes in bicircular representations} \label{seriesclasses} 

It is easy to see that if $M$ is a vertically 3-connected bicircular matroid, then for any element $e \in E(M)$, $M \del e$ has at most two non-trivial series classes, because deleting an edge from any graph representing $M$ may leave at most 
{two} 
vertices incident to exactly two edges, and so may leave at most two lines or balloons. 
We can show that this is also the case when $M$ is an excluded minor of rank at least six, at least in the case that $M \del e$ is essentially 3-connected. 

{Let $G$ and $H$ be a pair of graph representations for a connected bicircular matroid $M$. 
By Proposition \ref{elise}, an element $e \in E(M)$ is in a line or a balloon of $G$ if and only if it is in a line or balloon of $H$, and $e$ is in a non-trivial line or balloon of each of $G$ and $H$ if and only if $e$ is in a non-trivial series class of $M$. 
Say an element $e \in E(M)$ is \emph{involved in an $r$-operation} with respect to $G$ and $H$, or just \emph{involved} for short, if $H$ is obtained from $G$ via a sequence of rolling, rotation, replacement, or rearrangement operations and $e$ is an edge whose incidences are redefined in any of these operations. 
An edge is \emph{uninvolved} if it is not involved. 
Let us say that $G$ and $H$ have \emph{compatible vertex labelings} if their respective vertex sets are labelled such that: 
\begin{itemize}
\item each uninvolved edge has the same ends in $G$ as in $H$; 
\item if $G$ and $H$ are type 1 or 2, then $v$ is an apex vertex of $G$ if and only if $v$ is an apex vertex of $H$; 
\item if $G$ and $H$ are type 2, then $u$ is the rotation vertex of $G$ if and only if $u$ is the rotation vertex of $H$;  
\item for each non-trivial series class $X$ of $M$, the sets 
of internal vertices of $X$ in $G$ and in $H$ are the same; 
\item if $X$ is a line (resp.\ balloon) of $G$ that contains an involved edge, and $X$ is a line (resp.\ balloon) of $H$, then 
$X$ has the same vertices (resp.\ vertex) of attachment in $G$ as in $H$; 
\item if $X$ is a (possibly trivial) line in one of $G$ or $H$ while in the other graph $X$ is a balloon, then the vertex of attachment of the balloon is one of the vertices of attachment of the line in the other graph. 
\end{itemize}
Clearly, given any pair of graphs $G$ and $H$ representing a common bicircular matroid, compatible vertex labelings always exist. Thus we may assume that such a pair of graphs come equipped with compatible vertex labelings whenever necessary.}
Because vertex labels in our graphs carry no meaning nor have any significance other than as labels, the distinction between equality and isomorphism of two graphs with the same edge set that are identical after relabelling vertices is not relevant for us. 
We therefore say two graphs $G$ and $H$ are \emph{equal up to rolling, rotation, replacement, and rearrangement}, 
and write ``$G=H$ 
up to rolling, rotation, replacement, and rearrangement'', 
when $G$ may be obtained from $H$ via a sequence of rolling, rotation, replacement, or rearrangement operations, and relabelling of vertices. 

Let us call a vertex that is incident to exactly two edges \emph{deficient}. 

\begin{prop} \label{ernest} 
Let $G$ and $H$ be representations for a connected bicircular matroid $M$ with 
{compatible vertex labelings}. 
Let $v$ be a deficient vertex of $G$. 
Then $v$ is deficient in $H$ unless 
\begin{enumerate}[label={\upshape (\roman*)}]
\item $v$ is in a balloon $B$ of $G$ with $\deg_G(v) = 2$, while $H[B]$ is a balloon with $\deg_H(v) = 3$; 
\item $G$ and $H$ are type 1 representations of $M$, $v$ is the apex vertex in $G$ and $H$, $\deg_H(v) = 3$, and $G$ is obtained from $H$ by rolling a line away from $v$; or 
\item $G$ and $H$ are type 1 representations of $M$, $v$ is a vertex of degree 2 in a line $L$ of $G$ that has the apex vertex $u$ of $G$ as an end, $H$ is obtained by rolling $L$ away from $u$, and $\deg_H(v) = 3$. 
\end{enumerate} 
\end{prop}

\begin{proof} 
A rotation operation does not change the number of edges incident to any vertex; neither does a replacement operation applied to a line. 
A replacement operation applied to a balloon may change the number of edges incident to a vertex just as described in (i). 
A rollup operation may change the number of edges incident to the apex vertex $u$ of a type 1 representation and a vertex of the line other than the neighbour of $u$ in the line.  
Thus a rollup may cause a vertex to become or cease being deficient just as described in (ii) and (iii). 
\end{proof}


Let us call a representation of a balloon \emph{standard} if 
its vertex of attachment has degree 1 and the neighbour of its vertex of attachment has degree 3. 
Call a type 1 graph \emph{substandard} if 
every balloon has the apex as its vertex of attachment and every loop is incident to the (same) apex. 
Let us call a graph representation $G$ for a bicircular matroid \emph{standard} if 
\begin{itemize} 
\item each of its balloons are standard, and in addition, 
\item if $G$ is type 1 then $G$ is obtained from a substandard type 1 graph by rolling exactly one line $L$ away from its apex, where among all lines with the apex as a vertex of attachment $L$ has the greatest number of edges. 
\end{itemize}
By Theorem \ref{goose}, if $G$ is a non-standard representation, then a standard representation may be obtained from $G$ via rolling and replacement operations. 
Observe that a standard representation for a matroid $M$ has, among all graphs representing $M$, the least number of deficient vertices, provided the degree of the apex is at least three. 

We need a couple 
{of} 
simple facts before proceeding. 

\begin{prop} \label{bababa} 
\addcontentsline{toc}{subsubsection}{$N$ del $e$ is not a circuit}
Let $N$ be an excluded minor for the class of bicircular matroids, with rank at least 6, and let $e \in E(N)$. 
Then neither $N \del e$ nor $N/e$ is a circuit. 
\end{prop}

\begin{proof}
Suppose to the contrary that $N \del e$ is a circuit. 
Choose an element $s \in E(N)$; $N \del e / s$ is a circuit. 
Let $H$ be a graph representing $N/s$. 
Since $N/s \del e$ is a circuit of rank at least 5, $H \del e$ is a bicycle with at least 5 vertices. 
Since $H$ has at least five vertices, this implies $H$ has a vertex of degree 2, and so that $N/s$ has a cocircuit of size 2. 
But this implies that $N$ has a cocircuit of size 2, contradicting the fact that 
{by Lemma \ref{Nvert3},}
$N$ is vertically 3-connected. 

Now suppose $N/e$ is a circuit. 
Then either $N$ is a circuit or $e$ is a coloop of $N$. 
But circuits are bicircular and $N$ is vertically 3-connected, so this is not possible. 
\end{proof}

The following fact follows immediately from Propositions \ref{elise} and \ref{bababa}, along with the fact that deleting an element from 
a vertically 3-connected matroid 
leaves the matroid connected, 
while contracting an element leaves the matroid connected,  aside from the possibility of creating loops when contracting an element contained in a non-trivial parallel class. 

\begin{prop} \label{useful} 
Let $N$ be an excluded minor for the class of bicircular matroids, with rank at least 6. 
Let $e \in E(N)$ and let $S \subseteq E(N)-\{e\}$. 
Let $G$ be a representation for $N \del e$ or for $N/e$. 
Then $S$ is a series class of $N \del e$ or $N/e$, respectively, if and only if $G[S]$ is a line or a balloon. 
\end{prop}

We may now show that deleting an element from an excluded minor leaves at most two non-trivial series classes, subject to the assumption that the deletion remains essentially 3-connected. 
In fact, we prove a stronger statement. 

\begin{lem} \label{junniper} 
Let $N$ be an excluded minor for the class of bicircular matroids, with rank at least 6. 
Let $e$ be an element of $N$, and assume $N \del e$ is essentially 3-connected. 
Let $G$ be a standard representation for $N \del e$.  
Then $G$ has at most two deficient vertices. 
\end{lem}

\begin{proof} 
Suppose to the contrary that $G$ has more than two deficient vertices. 
Each deficient vertex is contained in a line or balloon of $G$. 
By Proposition \ref{useful}, the lines and balloons of $G$ correspond precisely to the non-trivial series classes of $N \del e$. 
So $N \del e$ has at least one non-trivial series class. 
Let $S_1, \ldots, S_k$ be the non-trivial series classes of $N \del e$. 
Let $S = \bigcup_i S_i$. 
Each series class $S_i$ induces either a line or balloon in $G$, so $G$ has at least one deficient vertex in each induced subgraph $G[S_i]$, $i \in \{1,\ldots, k\}$. 

\begin{claim} 
$k \leq 3$. 
\end{claim}

\begin{proof}[Proof of Claim]
Suppose $k \geq 4$. 
Choose an element $s \in S_4$, and let $H$ be a graph representing $N/s$. 
The matroid $N \del e / s$ is essentially 3-connected and  represented by $G/s$ and by $H \del e$. 
Thus by 
Lemma \ref{duckduck},  
$H \del e = G/s$ up to rolling, rotation, 
replacement, 
{and rearrangement}.
Since $N/s$ has no non-trivial series class, by Proposition \ref{useful} $H$ has no line nor balloon. 
Thus deleting $e$ from $H$ leaves at most two deficient vertices. 
This implies that $N \del e /s$ has at most two non-trivial series classes. 
But $G$ has at least four lines or balloons 
$H[S_i]$, 
$i \in \{1,\ldots,4\}$, so $G/s$ has at least three lines or balloons. 
That is, $N \del e /s$ has at least three non-trivial series classes, a contradiction. 
\end{proof}

\begin{claim} \label{icabod}
If $k=3$, then each non-trivial series class of $N \del e$ has size 2. 
\end{claim}

\begin{proof}[Proof of Claim]
Suppose to the contrary that $|S_1| \geq 3$. 
Let $s \in S_1$.  
Let $H$ be a graph representing $N/s$. 
As in the previous paragraph, $N \del e / s$ remains essentially 3-connected and is represented by both $G/s$ and $H \del e$. 
Thus possibly after rolling, rotation, 
replacement, 
{or rearrangement,} 
$H \del e = G/s$. 
As in the previous paragraph, since $N/s$ has no non-trivial series classes, $H$ has no deficient vertices, and so $H \del e$ has at most two deficient vertices. 
Thus $N \del e /s$ has at most two non-trivial series classes. 
But $G/s$ has three lines or balloons, so $N \del e /s$ has three non-trivial series classes, so this is a contradiction. 
\end{proof}

\begin{claim} \label{gwen} 
$N \del e$ has an element that is not in $S$. 
\end{claim}

\begin{proof}[Proof of Claim]
Suppose to the contrary that every element of $N \del e$ is contained in a non-trivial series class. 
By Proposition \ref{bababa}, $N \del e$ is not a circuit, so $N \del e$ does not consist of just one non-trivial series class. 
Because $N \del e$ is connected, $N \del e$ cannot consist of exactly two non-trivial series classes. 
So $N \del e$ has exactly three non-trivial series classes. 
By the previous claim, each class contains just two elements. 
This implies $N$ is a matroid on 7 elements, and so has rank either 6 or 7. 
But a rank 6 matroid on 7 elements is a circuit, and a rank 7 matroid on 7 elements is free, and both of these are bicircular.
\end{proof}

So choose an element $f \in E(N)-S$. 
Since $N$ is vertically 3-connected and has no parallel class of size greater than two, $N/f$ is connected, possibly aside from a component consisting of a loop. 
Let $H$ be a graph representing $N/f$ (where, if $N/f$ has a component consisting of a loop, then $H$ has a component consisting of a single vertex with a single incident balanced loop). 
As previously, since $N/f$ has no non-trivial series class, $H$ has no deficient vertex. 
But both $G/f$ and $H \del e$ represent $N \del e /f$ so, 
up to rolling, rotation, replacement, 
{and rearrangement,}
$G/f = H \del e$ 
(since by 
{assumption} 
$H$ has a balanced loop incident to an otherwise isolated vertex if necessary, and if so, then $f$ must be an unbalanced loop in $G$, one of two loops incident to a common vertex, so that when $f$ is contracted in $G$, a single balanced loop incident to an otherwise isolated vertex results; apply 
{Lemma \ref{duckduck}} 
to the matroid and graphs obtained by deleting the loop of $N \del e /f$). 

As $H$ has no deficient vertices, $H \del e$ can have at most two deficient vertices. 
Thus $N \del e /f$ has at most two non-trivial series classes. 
More, since $f \notin S$, contracting $f$ in $G$ cannot reduce the number of deficient vertices. 

Because $G$ is standard, and $f \notin S$, $G/f$ is standard. 
Thus among all graphs representing $N \del e /f$, $G/f$ has the least number of deficient vertices. 
Thus, since $H \del e$ has at most two deficient vertices, so does $G/f$ have at most two deficient vertices. 
\end{proof}

Lemma \ref{deficientdan} follows immediately from Proposition \ref{elise}, and Lemma \ref{junniper}. 

\begin{lem}\label{deficientdan} 
\addcontentsline{toc}{subsubsection}{only a few deficient vertices}
Let $N$ be an excluded minor, of rank at least 6, for the class of bicircular matroids. 
Suppose $N \del e$ is essentially 3-connected with  representation $G$. 
\begin{enumerate}[label={\upshape (\roman*)}]
\item $N \del e$ has at most two non-trivial series classes. 
\item If $N \del e$ has just one non-trivial series class $S$, then 
\begin{itemize} 
\item $|S| \leq 4$, and 
\item if $G[S]$ is a line, then $|S| \leq 3$, unless $G$ is type 1 and $G[S]$ has an apex as one end. 
\end{itemize} 
\item If $N \del e$ has two non-trivial series classes, $S_i$ \textup{(}$i \in \{1,2\}$\textup{)}, then 
\begin{itemize} 
\item $|S_i| \leq 3$, and 
\item if $G[S_i]$ is a line, then $|S_i| = 2$, unless $G$ is type 1 and $G[S_i]$ has an apex as one end. 
\end{itemize} 
\end{enumerate}
\end{lem}

Using Lemmas \ref{junniper} and \ref{deficientdan} we can now show that for an excluded minor $N$ of rank at least seven, if a graph representing $N \del e$ is type 1 or type 2, then its apex vertex is unique. 

\begin{prop} \label{pedro} 
Let $N$ be an excluded minor of rank at least seven. 
Assume $N \del e$ is essentially 3-connected and that $G$ is  a type 1 or type 2 representation for $N \del e$. 
Then $G$ has a unique apex vertex. 
The apex vertex of $G$ is not an internal vertex of a line or balloon of $G$. 
\end{prop}

\begin{proof} 
Suppose to the contrary that $G$ is type 1 or type 2 and does not have a unique apex. 
Then $G$ has the structure described in Proposition \ref{serabande}, and has at least seven vertices. 
By Lemma \ref{deficientdan}, $N \del e$ has at most two non-trivial series classes. 

Suppose first $G$ is type 1. 
If $N \del e$ has just one non-trivial series class, then it has size at most four. 
But this implies $G$ has at most five vertices, a contradiction. 
If $N \del e$ has two non-trivial series classes, then each has size two or three. 
But this implies $G$ has at most six vertices, a contradiction. 

Now suppose $G$ is type 2. 
Then $G$ is a standard representation of $N \del e$, and so has at most two deficient vertices, by Lemma \ref{junniper}. 
But this implies $G$ has at most six vertices, a contradiction. 

Let $v$ be the apex of $G$. 
Clearly, $\deg_G(v)$ must be at least three, so $v$ cannot be an internal vertex of a line or balloon of $G$. 
\end{proof}

\subsection{Repairing connectivity} 

Let $N$ be an excluded minor of rank at least seven. 
The 
goal of this section is to find an element $e \in E(N)$ for which $N \del e$ is not only essentially 3-connected, but also represented by a 2-connected graph. 

\begin{lem} \label{B5} 
\addcontentsline{toc}{subsubsection}{Repairing lines \& balloons with a deletion}
Let $N$ be an excluded minor for the class of bicircular matroids, of rank at least seven. 
Assume $N$ has an element $e$ so that $N \del e$ is essentially 3-connected, and let $G$ be a graph representing $N \del e$. 
Assume $G$ is not 2-connected, but $G$ has an edge $f$ not contained in a line or balloon such that $\co(G) \del f$ is 2-connected. 
Let $H$ be a graph representing $N \del f$ 
{and assume $G \del f$ and $H \del e$ have compatible vertex labelings.}
Then 
\begin{enumerate}[label={\upshape (\roman*)}]
\item  $G \del f = H \del e$ up to rolling, rotation, and replacement, 
\item 
each line or balloon of $G$ has an internal vertex that is an end of $e$ in $H$, and 
no line or balloon of $H$ is contained in a line or balloon of $G$, 
\item $N \del f$ is essentially 3-connected and $H$ is 2-connected up to rolling, and, 
\item no balloon of $G$ has more than 3 edges. 
\end{enumerate} 
\end{lem}

\begin{proof} 
Since $G$ is not 2-connected but $N \del e$ is essentially 3-connected, $G$ has a balloon. 
{Thus $G$ is not in $\mathscr G$. 
Since $f$ is not contained in this balloon, neither is $G \del f$ in $\mathscr G$. 
Both $G \del f$ and $H \del e$ represent $N \del \{e,f\}$, so by 
Theorem \ref{goose}, statement (i) holds.} 
By Proposition \ref{useful} and 
{Lemma} 
\ref{deficientdan}, $G$ has at most one other balloon or line. 
Let $S$ be the set of edges contained in a balloon or line of $G$. 
{Then $S$ consists of at most two non-trivial series classes of $N \del e$, and because we choose $f$ not contained in a line or balloon of $G$, $f \notin S$.} 

\begin{claim} 
No subset $R$ of $S$ is a non-trivial series class of $N \del f$. 
\end{claim}

\begin{proof}[Proof of Claim]
Suppose there is such a subset $R \subseteq S$. 
{Then $R$ is a set of elements in series in $N \del e$, as well as in $N \del f$.}
By Proposition \ref{useful}, $H[R]$ is a balloon or line. 
Write $Q = E(H)-R$. 
Then $e \in Q$ and $(Q,R)$ is either a proper 1-separation of $H$ where $R$ is a balloon or a proper 2-separation of $H$ where $R$ is a line.
In either case, $(Q,R)$ is a vertical 2-separation of $N \del f$. 
{We now show that there is a circuit of $N$ contained in $Q \cup \{f\} - \{e\}$ that contains $f$. For suppose not. Then every circuit of $N \del e$ containing $f$ contains an element of $R$, and so contains every element of $R$. The matroid $N \del e$ is essentially 3-connected, so such a circuit exists. Thus $f$ is contained in the series class of $N \del e$ containing $R$. But then by Proposition \ref{useful}, $f$ is contained in a line or balloon of $G$, contrary to our choice of $f$. Thus there is a circuit contained in $Q \cup \{f\}$ containing $f$. Hence $f \in \cl_N(Q)$.}
But this implies that $(Q \cup f, R)$ is a vertical 2-separation of $N$, contrary to Lemma \ref{Nvert3}.
\end{proof}

By the claim neither $S$, nor any subset of $S$, is a balloon or line in $H$.  
Since $H \del e = G \del f$ up to rolling, rotation, and replacement, this implies statement (ii): if $S'$ is a line or balloon of $G$, then in $H$ edge $e$ has an internal vertex of $S'$ as an end, and no line or balloon of $H$ is contained in a line of balloon of $G$. 

Now suppose $H$ has a proper 1-separation $(A,B)$. 
Since $\co(G) \del f$ is 2-connected, and $G \del f = H \del e$ up to rolling, rotation, and replacement, 
up to relabelling the sides of the separation, either 
\begin{enumerate} 
\item $V(A) \cap V(B)$ is the vertex of attachment of a balloon $S$ in $G$, $A = S \cup e$, and $e$ has both ends in $V_G(S)$, or 
\item 
$H \del e$ is a type 1 representation for $N \del \{e,f\}$, 
$G$ has a balloon $S$, 
$G \del f$ is type 1 with a line $A$, and 
$H$ is obtained from $G$ by unrolling $S$ to its apex, 
adding $e$ as an edge 
with at least one end an internal vertex of $G[S]$, 
and rolling $A$ away from the apex.  
\end{enumerate}
In the first case, since $f \in \cl_N(B)$, the 2-separation $(A,B)$ of $N \del f$ extends to a 2-separation $(A, B \cup f)$ of $N$, a contradiction. 
Thus the second case holds; that is, statement (iii) holds. 

Finally, suppose a balloon $S$ in $G$ has more than three edges, and suppose $G[S]$ has vertex of attachment $x$. 
Then $G[S]$ has at least two deficient vertices. 
But by Lemma \ref{junniper}, $G$ has at most two deficient vertices after possibly applying a replacement operation, so $|S|=4$. 
Let $A=E(G)-S$; $(A,S)$ is a proper 1-separation of $G$ and a vertical 2-separation of $N \del e$. 
Since no subset of $S$ forms a line or balloon in $H$, in $H$ edge $e$ has both ends in $V_H(S)$. 
But this implies that $e \in \cl_N(S)$, so the vertical 2-separation $(A, S)$ of $N \del e$ extends to a vertical 2-separation $(A, S \cup e)$ of $N$, contradicting the fact that $N$ is vertically 3-connected. 
\end{proof}

We describe outcome (ii) of Lemma \ref{B5} by saying that the edge $e$ in $H$ \emph{repairs} a line or balloon of $G$. 

\begin{lem} \label{eggbert} 
\addcontentsline{toc}{subsubsection}{each series class a line}
Let $N$ be an excluded minor for the class of bicircular matroids, with rank at least seven. 
If $N$ has an element $e$ such that $N \del e$ is essentially 3-connected and type 3, 
then $N$ has an element $e'$ such that $N \del e'$ is essentially 3-connected, type 3, and represented by a 2-connected graph.
\end{lem}

\begin{proof} 
Let $e \in E(N)$, assume $N \del e$ is essentially 3-connected, and let $G$ be a type 3 graph representing $N \del e$. 
If $G$ is 2-connected we are done, so assume not. 
Since $N \del e$ is essentially 3-connected, $G$ has a balloon $S_1$ that is a series class of $N \del e$. 
By Lemma \ref{deficientdan}, $N \del e$ has at most two non-trivial series classes. 
If it exists, let $S_2$ be the second non-trivial series class of $N \del e$; otherwise, let $S_2 = \emptyset$. 
Let $s_1 \in S_1$ and, if $S_2$ is non-empty, let $s_2 \in S_2$, and let $G'$ be the graph obtained by contracting all elements in $S_1$ and $S_2$ other than $s_1$ and $s_2$ 
Then $G'$ is 2-connected with minimum degree at least 3. 
Note that $s_1$ is a loop in $G'$. 

\begin{claim} \label{frank} 
$G'$ has an edge $e'$ distinct from $s_1$ and, if 
{$S_2$ is non-empty then distinct from each of $s_1$ and $s_2$},
such that $G' \del e'$ is 2-connected. 
\end{claim}

\begin{proof}[Proof of claim]
If $G'$ is 3-connected, then choose any edge $e' \in E(G')$ other than $s_1$ or $s_2$: $G' \del e'$ is 2-connected, so the claim holds. 
So assume $G'$ is not 3-connected. 
If $S_2$ is empty, or if there is a proper 2-separation of $G'$ with both $s_1$ and $s_2$ in the same side, then let $(A,B)$ be a proper 2-separation of $G$ with $s_1$ in $A$ and, if $S_2$ is non-empty, both $s_1, s_2$ in $A$, and with $B$ minimal. 
By Proposition \ref{xavier}, there is an edge $e' \in B$ such that $G' \del e'$ remains 2-connected. 

So now assume $S_2$ is non-empty and every proper 2-separation of $G'$ has $s_1$ and $s_2$ in different sides. 
Then 
for any edge $f \in E(G')$, if $G' \del f$ has a proper 1-separation, then $s_1, s_2$ are in opposite sides of this separation. 
Thus to check that $G' \del f$ remains 2-connected, we just need check that there remains a pair of internally disjoint paths in $G' \del f$ linking the end of $s_1$ and the ends of $s_2$. 

{
Let $(A,B)$ be a 2-separation of $G'$ with $s_1 \in A$ and $A$ minimal, and let $(C,D)$ be a 2-separation of $G'$ with $s_2 \in D$ and $D$ minimal. 
By Proposition \ref{xavier} and its proof, there is an edge $e'$ as required, unless $|V(A)-V(B)| = 1 = |V(D)-V(C)|$, and each of $A$ and $D$ have just three edges, and each of $s_1$ and $s_2$ are loops (the case shown in the graph at left in Figure \ref{manny1}). 
So assume that this is the case (Figure \ref{Findingeprimegood}). 
\begin{figure}[tbp] 
\begin{center} 
\includegraphics[scale=0.95]{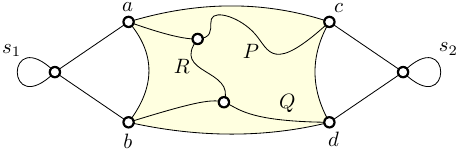}
\end{center} 
\caption{{Finding $e'$ when every 2-separation of $G'$ has $s_1$ and $s_2$ in different sides.}}
\label{Findingeprimegood} 
\end{figure}   
Because $G'$ is 2-connected, there is a pair of disjoint paths $P, Q$ linking $V(A) \cap V(B)$ and $V(C) \cap V(D)$. 
There is a path $R$ linking a vertex in $P$ and a vertex in $Q$, as otherwise $G'$ would have a proper 2-separation with $s_1$ and $s_2$ on the same side, contrary to assumption. 
(Such a 2-separation would exist even in the case that both $P$ and $Q$ are trivial, where $a=c$ and $b=d$, though in this case the non-existence of path $R$ would also imply that $N \del e$ has more than two non-trivial series classes, contradicting Lemma \ref{deficientdan}.) 
Let $e'$ be an edge in $R$. 
The paths $P$ and $Q$ remain in $G' \del e'$, so $G' \del e'$ remains 2-connected. }
\end{proof}

{Let} 
$e' \in E(G') - \{s_1, s_2\}$ 
(respectively, let $e' \in E(G') - \{s_1\}$ if $S_2$ is empty) 
be an edge 
{whose existence is guaranteed by the preceding claim.} 
{Then} 
$G' \del e'$ remains 2-connected. Let $H$ be a graph representing $N \del e'$ 
{and assume that $V(G)$ and $V(H)$ are labelled so that $H \del e$ and $G \del e'$ have compatible vertex labelings}. 
By Lemma \ref{B5}, $H$ is 2-connected up to a rolling operation. 
Thus $N \del e'$ is essentially 3-connected. 
If $H$ is type 3 we are done. 

So assume $H$ is a type 1 or type 2 graph representation for $N \del e'$. 
We may assume $H$ is a substandard representation for $N \del e'$ (that is, if $H$ is type 1, then $H$ does not have a balloon that may be rolled up to its apex), so $H$ is 2-connected. 
By Lemma \ref{pedro} $H$ has a unique apex vertex $v$. 
Observe that $N \del \{e,e'\}$ is essentially 3-connected. 
Since 
{$N \del \{e,e'\}$ is represented by both $H \del e$ and $G \del e'$, and $G \del e'$ is not in $\Ggg$ (because $G \del e'$ has a balloon),}
$H \del e = G \del e'$ up to rolling, rotation, and replacement. 
{Thus}
$G \del e'$ is type 1 or 2 with apex $v$. 
By Proposition \ref{pedro}, 
$v$ is not an internal vertex of a balloon or line in $H$. 

\begin{claim} 
Other than its balloon $S_1$, $G$ has no other balloon nor a line. 
\end{claim}

\begin{proof}[Proof of Claim]
Suppose that $G$ has a second line or balloon $S_2$. 
By Lemma \ref{B5}, $e$ repairs both of $S_1$ and $S_2$ in $H$. 
Then not both $S_1$ and $S_2$ are balloons with $v$ as their vertex of attachment (if so, $H$ would have $v$ as a cut-vertex, but $H$ is 2-connected). 
Since $v$ is not an internal vertex of a balloon or a line, this implies that $H$ is type 3, a contradiction. 
\end{proof}

\begin{claim} 
$H$ has an edge $f \neq e$, incident to $v$, such that $H \del f$ remains 2-connected. 
\end{claim}

\begin{proof}[Proof of Claim] 
Since $H \del e = G \del e'$ up to rolling, rotation, and replacement, $V_G(S_1) \cup \{v\} = V_H({S_1}) \cup \{v\}$; let us denote this set by $V_{S_1}$, and by $V_{S_1}^-$ the set of vertices obtained by removing the vertex of attachment of ${S_1}$ from $V_{S_1}$. 
Now because $H$ is type 1 or 2 and 2-connected, both ends of $e$ in $H$ are in $V_{S_1}$. 
By Lemma \ref{B5}, $|{S_1}| \leq 3$, so $|V_{S_1}^-| \leq 3$. 

Since $G$ is type 3, but $G \del e'$ is not, $v$ is not an end of $e'$ in $G$. 
Thus the fact that $G \del e' = H \del e$ up to rolling, rotation, and replacement, 
{and that $e$ repairs $S_1$,}
implies that the only vertices of $H$ that are possibly of degree less than three are the ends of $e'$ in $G$. 
Let us denote by $T$ the tree $(G-V_{S_1}^-) \del e' = H-V_{S_1}^-$. 
Since $T$ 
has at least four vertices, of which at most two are of degree 2 in $H$, connectivity now implies that there is either a leaf of $T$ with at least two incident edges whose other end is $v$, or there is a non-leaf of $T$ with an incident edge whose other end is $v$, or the vertex of attachment of ${S_1}$ is a leaf of $T$ and so must be an end of an edge whose other end is $v$. 
Choose such an edge $f$: 
$H \del f$ remains 2-connected.  
\end{proof}

Let $K$ be a graph representing $N \del f$, 
{with compatible vertex labelling}. 
Because $H \del f$ is 2-connected, $N \del \{e',f\}$ is essentially 3-connected, so $N \del f$ is essentially 3-connected. 
By our choice of $f$, $G \del f$ remains type 3, 
{and as $G \del f$ has a balloon, $G \del f$ is not in $\Ggg$.} 
Therefore, up to rolling, rotation, and replacement, 
\begin{itemize} 
\item $K \del e = G \del f$, so $K$ is type 3 and $e'$ has the same ends in $K$ as in $G$, and 
\item $K \del e' = H \del f$; 
the proof of statement (iii) of Lemma \ref{B5} shows that $e$ must have the same ends in $K$ as in $H$. 
\end{itemize} 
Thus $K$ is a 2-connected type 3 representation for $N \del f$. 
\end{proof}

\section{Finding $X$} \label{findingX}

Given an excluded minor $N$, of rank at least 
{ten}, 
the goal of this section is to find a triple of elements $X = \{a,b,c\} \subseteq E(N)$ and a bicircular twin $M$ for $N$ relative to $X$, such that 
\begin{itemize} 
\item $X$ spans at least five vertices in each representation for $M$, and 
\item for each subset $Z$ of $X$, $N/Z$ is vertically 3-connected. 
\end{itemize}

\subsection{Contractible edges}

We use the Edmonds-Gallai decomposition of graphs without a perfect matching. 
For any subset $U$ of vertices of a graph $G$, denote by $N(U)$ the set of \emph{neighbours of $U$}, $N(U) = \{ y \in V(G) - U :$ there is an edge $xy \in E(G)$ 
{with $x \in U$}$\}$. 
A graph is \emph{hypomatchable} if for every vertex $v \in V(G)$, $G-v$ has a perfect matching. 

\begin{thm}[Edmonds-Gallai {\cite[Theorem 3.2.1]{MR859549}}] \label{Ed-Gallai}
Let $G$ be a graph and let $A \subseteq V(G)$ be the set of vertices $v$ such that $G$ has a maximum matching that does not cover $v$. 
Set $B = N(A)$ and $C = V(G) - (A \cup B)$. 
Then 
\begin{enumerate}[label={\upshape (\roman*)}]
\item Every odd component $H$ of $G-B$ is hypomatchable and satisfies $V(H) \subseteq A$. 
\item Every even component $H$ of $G-B$ has a perfect matching and satisfies $V(H) \subseteq C$. 
\item For every set $U \subseteq B$, the set $N(U)$ contains vertices in more than $|U|$ odd components of $G-B$. 
\end{enumerate} 
\end{thm}

Call the set $B$ defined in the Edmonds-Gallai Theorem the \emph{barrier set} of $G$. 
A
{non-loop}
edge $e$ in a 2-connected graph $G$ is \emph{contractible} if $G/e$ remains 2-connected. 

\begin{prop} \label{thecontractiblegraph} 
\addcontentsline{toc}{subsubsection}{The contractible graph remains 2-connected}
Let $G=(V,E)$ be a 2-connected graph on at least 4 vertices, and let $S \subseteq E$ be the set of contractible edges of $G$. 
Then $(V,S)$ is 2-connected. 
\end{prop} 

\begin{proof} 
If $G$ is 3-connected, then every edge is in $S$ and the result holds. 
So assume $G$ has a proper 2-separation. 
We proceed by induction on the number of vertices of $G$. 
The result clearly holds when $|V(G)|=4$, so assume $|V(G)| > 4$ and that the result hold for graphs with less than $|V(G)|$ vertices. 
Let $(A,B)$ be a proper 2-separation of $G$ with $B$ minimal. 
{We claim that}
$B \subseteq S$. 
{To see this, let $V(A) \cap V(B) = \{x,y\}$, and observe that by minimality there is no $x$-$y$ edge in $B$ (such an edge may be removed from $B$ and added to $A$, so $B$ would not be minimal). 
Now suppose $e \in B$ but $e$ is not contractible. 
Then $G/e$ has a cut vertex $a$. 
This implies that in $G$ either $e$ has distinct endpoints $u,v$ that are identified to produce the vertex $a$ in $G/e$, or that $e$ is a loop in $G$, say incident to the vertex $b$ and that $G$ has an $a$-$b$ edge that becomes a loop in $G/e$ incident to $a$. 
In the first case, $G$ has a 2-separation $(A',B')$ with $V(A') \cap V(B') = \{u,v\}$ with $B' \subset B$, contradicting the minimality of $B$. 
In the second case, $G$ has a 2-separation $(A',B')$ with
$V(A') \cap V(B') = \{a,b\}$ and $B' \subseteq B$, again contradicting the minimality of $B$.}
 
Let $G' = (V',E')$ be the graph obtained from $G$ by replacing $B$ with a single edge linking the two vertices in $V(A) \cap V(B)$. 
Let $S'$ be the set of contractible edges of $G'$. 
By induction, $(V',S')$ is 2-connected. 
Let $x$, $y$ be the vertices in $V(A) \cap V(B)$. 
Let $R$ be the set of edges of $G$ that are either loops incident to $x$ or $y$ or have both $x$ or $y$ as ends. 
While no edge in $R$ is contractible in $G$, every edge in $R$ is contractible in $G'$. 
By minimality, no edge in $R$ is in $B$. 
Thus $S = (S' - R) \cup B$, and so $(V,S)$ is 2-connected. 
\end{proof} 

We denote by $\si(G)$ the graph obtained from a graph $G$ by replacing every set of parallel edges with a single edge and removing all loops. 
Denote by $K_{2,n}'$ the graph obtained by adding an edge linking the two degree $n$ vertices of $K_{2,n}$. 

\begin{prop} \label{marty}
\addcontentsline{toc}{subsubsection}{The contraction graph has a matching}
Let $G$ be a 2-connected graph on at least 6 vertices, and let $S$ be the set of contractible edges of $G$. 
Then $G[S]$ has a matching of size 3 unless $\si(G)$ is $K_{2,n}$ or $K_{2,n}'$. 
\end{prop}

\begin{proof} 
Suppose $G[S]$ does not have a matching of size 3. 
Then its barrier set $B$ has size at most 2, 
{by Hall's Theorem and Theorem \ref{Ed-Gallai}}. 
If $|B| = 1$, then $G[S]$ has the single vertex in $B$ as a cut vertex, contradicting Proposition \ref{thecontractiblegraph}. 
So $|B| = 2$, and $\si(G[S])$ is isomorphic to $K_{2,n}$, where $n \geq 4$. 
Let $e'$ be the edge of $G$ linking the two vertices of degree $n+1$ of $\si(G[S])$, if it exists. 
It is easy to see that if $G$ had an edge $e$, other than $e'$, that is not in $S$, then $e$ would nevertheless be contractible
a contradiction. 
Thus $\si(G)$ is $K_{2,n}$ or $K_{2,n}'$. 
\end{proof} 

\begin{prop} \label{constance} 
\addcontentsline{toc}{subsubsection}{subsets of contractible matchings are contractible}
Let $G$ be a 2-connected graph and let $X$ be a matching in $G$. 
If $G/X$ is 2-connected, then $G/Z$ is 2-connected for every $Z \subseteq X$. 
\end{prop} 

\begin{proof} 
Suppose to the contrary that $Z \subseteq X$ and $G/Z$ has a 1-separation $(A,B)$, say with $V(A) \cap V(B) = \{x\}$. 
Since $G$ does not have a cut vertex, a pair of vertices $u, v \in V(G)$ are identified in $G/Z$ as the vertex $x$, by contraction of an $u$-$v$ path contained in $Z$. 
But $Z$ is a matching, so this path consists of a single $u$-$v$ edge. 
Since $x$ is not a cut vertex of $G/X$, one of $A$ or $B$ has all of its edges contained in $X - Z$. 
But this implies that at least two edges in $X$ are incident to either $u$ or $v$, contrary to the fact that $X$ is a matching. 
\end{proof} 

\begin{prop} \label{sax} 
\addcontentsline{toc}{subsubsection}{a matching of contractible edges is contractible}
Let $G$ be a 2-connected graph and let $X$ be a matching of three edges in $G[S]$, where $S$ is the set of contractible edges of $G$. 
Then $G/Z$ is 2-connected for every subset $Z \subseteq X$. 
\end{prop}

\begin{proof} 
By Proposition \ref{constance}, we just need to show that $G/X$ is 2-connected. 
Let $X = \{a,b,c\} \subseteq E(G)$. 
Suppose for a contradiction that $G/X$ has a proper 1-separation. 
Because $a \in S$, $G/a$ does not have a proper 1-separation. 
So there must be a $k$-separation $(A,B)$ of $G/a$, where $k \in \{2,3\}$, 
such that 
{either (i) $k=2$ and one of $b$ or $c$ has the vertices in $V(A) \cap V(B)$ as its endpoints, or (ii)} 
$G[\{b,c\}]$ is a path 
{in $G/a$}
containing $V(A) \cap V(B)$. 
{In the first case (i), since both $b$ and $c$ are contractible in $G$ one of the vertices in $V(A) \cap V(B)$ must be an endpoint of $a$ in $G$. But this is contrary to our choice of $a, b, c$ as a matching. The second case (ii)} 
implies that $b$ and $c$ share an end in $G/a$. 
Since $b$ and $c$ do not share an end in $G$, this implies $a$ shares an end with each of $b$ and $c$. 
But 
{again} 
this is not the case, as $X$ is a matching in $G$. 
\end{proof}

\subsection{Staying type 3} 

When constructing our twin for an excluded minor, relative to a set $X$, we would like not only to remain vertically 3-connected each time a subset of $X$ is contracted, but also remain of the same type with each contraction. 
In this subsection, we consider type 3 representations. 
In the next, we consider types 1 and 2. 

\begin{prop} \label{baritone} 
\addcontentsline{toc}{subsubsection}{contract a matching, stay type 3}
Let $X$ be a matching of three edges in a 2-connected type 3 graph $G$ with at least seven vertices, such that $G/Z$ is 2-connected for every subset $Z$ of $X$. 
Suppose that there is a subset $Y$ of $X$ such that $G/Y$ is no longer type 3. 
Then there is a set $X'$ of three 
{non-loop} 
edges spanning at least five vertices for which $G/Z$ is 2-connected and type 3, for every subset $Z$ of $X'$. 
\end{prop}

\begin{proof} 
Let $X = \{a,b,c\}$. 
After possibly permuting the labels of edges $a$, $b$, $c$, one of the following holds. 
\begin{enumerate}
\item There is a vertex $x$ such that $G-x$ is type 1 (resp.\ type 2) with apex vertex $v$, and $a$ has ends $v,x$. 
\item There is a vertex $x$ such that $(G-x) \del b$ is type 1 with apex $v$, $a$ has ends $v, x$, and 
$b$ is in a parallel pair 
{in $G$}. 
\item $G \del a$ is type 2 with apex $v$ and rotation vertex $u$, $a$ is incident to $v$ and is in a parallel pair $\{a,f\}$, and $f$ is in the line of $G \del a$ whose ends are $v$ and $u$. 
\item $G$ has a vertex $v$ such that $G-v$ is acyclic aside from exactly two loops that are not incident to the same vertex, and $a$ has ends $v$ and $x$, where $x$ is incident to one of the loops of $G-v$. 
\item $G$ has a vertex $v$ such that $G-v$ is acyclic aside from exactly one loop and one pair of parallel edges, $a$ has ends $v, x$, where $x$ is incident to the loop of $G-v$, and $b$ is one of the edges in the parallel pair. 
\item Edge $a$ is in a parallel pair; $G \del a$ is type 1 with apex $v$ and has no loop that is not incident to $v$. 
\end{enumerate} 

Let $S$ be the set of contractible edges of $G$. 
In each case, we find a triple $\{a',b',c'\}$ of edges as desired in $G[S]$. 

In cases (1) and (2) there is a cycle $C$ in $G-v$ (resp.\ in $G-\{u,v\}$, where $u$ is the rotation vertex when $G-x$ is type 2). 
 If $|C| \geq 5$, we may choose $a', b', c'$ as required from the edges of $C$. 
If $|C| \leq 4$, then because $G$ has at least 7 vertices, we can choose $a'$ from among the edges of $C$, $b'$ from among those edges incident to a vertex in $C$ that do not have $v$ as an end, and $c'$ from the set of edges incident to $v$ that are not incident to a vertex in $C$ nor an end of $b$. 
Our rank and connectivity assumptions guarantee that such edges exist. 

In the remaining cases, all 
{non-loop} 
edges in $G-v$ are in $S$, and $G-v$ consists of a tree, after possibly removing some edge or edges that we wish to preserve in $G/X'$ in order to ensure that $G/X'$ remains type 3; namely, an edge in a parallel pair incident to the vertex of rotation when $G \del a$ is type 2, loops not incident to $v$, or an edge from a parallel pair. 
It is easy to see that in each case we can find a triple $a',b',c'$ in this tree, as desired. 
\end{proof}

\begin{lem} \label{bill} 
\addcontentsline{toc}{subsubsection}{finding contractible edges in $G$}
Let $N$ be an excluded minor for the class of bicircular matroids, of rank at least 
{nine}. 
Let $e \in E(N)$ and assume that $G$ is a 2-connected type 3 representation for $N \del e$. 
Then there are 
{non-loop} 
edges $a,b,c \in E(G)$ spanning at least five vertices such that for every subset $Z \subseteq \{a,b,c\}$, $G/Z$ remains 2-connected and type 3. 
\end{lem}

\begin{proof}
Let $S$ be the set of contractible edges of $G$. 
If $G[S]$ has a matching $\{a,b,c\}$ of size 3, then by Proposition \ref{sax}, $G/Z$ is 2-connected for every $Z \subseteq \{a,b,c\}$. 
If $G/Z$ remains type 3 for every $Z \subseteq \{a,b,c\}$, we are done. 
Otherwise, by Proposition \ref{baritone}, there are three edges $a',b',c'$ as desired, and we are done. 

So assume that $G[S]$ does not have a matching of size 3. 
By Proposition \ref{marty}, $\si(G)$ is $K_{2,n}$ or $K_{2,n}'$. 
Let $u_1$, $u_2$ be the two vertices of degree $n$ or $n+1$ in $\si(G)$, and let $W = \{x_1, \ldots, x_n\}$ be the set of vertices of degree 2 in $\si(G)$. 
Our rank assumption implies that $n \geq 7$. 
By Lemma \ref{deficientdan}, at most two of $x_1, \ldots, x_n$ have degree 2 in $G$, so there are at least five vertices in $W$ of degree at least three. 
Because $G$ is type 3, up to relabelling $u_1$ and $u_2$, one of the following holds:  
\begin{itemize}
\item there are at least two loops in $G[W]$; 
\item there is just one loop incident to a vertex in $W$, but there is a loop incident to each of $u_1$ and $u_2$; 
\item there is just one loop incident to a vertex in $W$, a loop incident to $u_1$, and a vertex in $W$ with at least two edges linking it to $u_2$; 
\item there is just one loop in $G$, incident to a vertex in $W$, there is a vertex in $W$ with at least two edges linking it with $u_1$, and 
a vertex in $W$ with at least two edges linking it with $u_2$; or 
\item $G$ has no loops, and 
\begin{itemize} 
\item there are either at least two vertices in $W$ sending two edges to $u_1$ or there is a vertex in $W$ sending at least three edges to $u_1$, and 
\item there are either at least two vertices in $W$ sending two edges to $u_2$ or there is a vertex in $W$ sending at least three edges to $u_2$. 
\end{itemize} 
\end{itemize}
It is straightforward to check that in any case a triple of edges exists as required. 
\end{proof}

Observe that the property of being type 1 or type 2 (equivalently, the property of not being type 3) is closed under minors: that is, if $G$ is a type 1 graph, then every minor of $G$ is type 1, while if $G$ is a type 2 graph, then every minor of $G$ is either type 2 or type 1. 
This easy observation will be used more than once. 

\begin{prop} \label{sym29} 
\addcontentsline{toc}{subsubsection}{type 1/2 is minor-closed}
Let $M'$ be a connected minor of a connected bicircular matroid 
$M$. 
If $M'$ is type 3 then $M$ is type 3. 
\end{prop}

\begin{proof} 
Let $G'$ be a graph representing $M'$ and let $G$ be a graph representing $M$. 
Let $C, D \subseteq E(M)$ such that $M/C \del D = M'$. 
By Theorem \ref{goose}, 
$G/C \del D = G'$, up to rolling, rotation, and replacement, 
{unless one of these graphs is in $\Ggg$. 
If so, then $M'$ is essentially 3-connected, and so by Lemma \ref{duckduck}, $G/C \del D = G'$, up to rearrangement and replacement. 
In either case,} 
since $G'$ is type 3, $G$ must also be type 3. 
\end{proof} 

\begin{lem} \label{B6} 
\addcontentsline{toc}{subsubsection}{Repairing lines \& balloons with a contraction}
Let $N$ be an excluded minor for the class of bicircular matroids. 
Let $e \in E(N)$ and assume that $G$ is a 2-connected type 3 representation for $N \del e$. 
Let $Z$ be a set of at most three 
{non-loop} 
edges for which $G/Z$ remains 2-connected and type 3, and let $H$ 
be a graph representing $N/Z$. 
Then $H$ is 2-connected and type 3, and $N/Z$ is vertically 3-connected. 
\end{lem}

\begin{proof} 
Since $G/Z$ is 2-connected, $N \del e / Z$ is essentially 3-connected. 
The matroid $N \del e / Z$ is represented by both $G/Z$ and $H \del e$. 
Thus by Lemma \ref{duckduck}, $G/Z = H \del e$ up to rolling, rotation, 
replacement, and rearrangement. 
Since $G/Z$ is type 3, rolling and rotation are 
irrelevant, and $G/Z = H \del e$ up to replacement 
{and rearrangement}. 
{Thus $H \del e$ is type 3, and so}
$H$ is type 3. 

Now suppose $H$ has a proper 1-separation $(A,B)$, say with $V(A) \cap V(B) = \{x\}$ and $e \in A$. 
Then neither $A$ nor $B$ is a parallel class of $N/Z$. 
Since $N/Z$ has no non-trivial series class, 
$(A,B)$ is an essential 2-separation of $N/Z$. 
Since $N/Z \del e$ is essentially 3-connected, by Proposition \ref{cam}, $A-e$ is a balloon in $H \del e$. 
But $H \del e = G/Z$ up to replacement and rearrangement, and $G/Z$ is 2-connected, so this is impossible. 
Thus $H$ is 2-connected. 
Since $N/Z$ is represented by a 2-connected graph, $N/Z$ is essentially 3-connected. 
But $N/Z$ has no non-trivial series class, so $N/Z$ is vertically 3-connected. 
\end{proof}

Our main objective for this section now follows from Lemmas \ref{bill} and \ref{B6}: 

\begin{lem} \label{billbill} 
\addcontentsline{toc}{subsubsection}{finding contractible edges in $N$}
Let $N$ be an excluded minor for the class of bicircular matroids, of rank at least 
{ten}, and let $e$ be an element such that $N \del e$ is essentially 3-connected and represented by a 2-connected type 3 graph $G$. 
Then there is a bicircular twin $M$ for $N$ relative to a set $X$ of size three, represented by a graph $H$ in which $|V_H(X)| \geq 5$, 
{no element in $X$ is a loop in $H$,} 
and such that for every subset $Z \subseteq X$, 
$N/Z$ is vertically 3-connected and $H/Z$ is 2-connected and type 3. 
\end{lem}

\begin{proof} 
By Lemma \ref{bill}, there 
{is a set $X$ of}
three 
{non-loop} 
edges $a,b,c \in E(G)$, spanning at least five vertices, such that for every subset $Z \subseteq \{a,b,c\}$, $G/Z$ is 2-connected and type 3. 
Thus for each subset $Z \subseteq \{a,b,c\}$, $N \del e /Z$ is essentially 3-connected
{and}
type 3, and so 
by Lemma \ref{duckduck} is represented uniquely up to replacement 
{and rearrangement} 
by $G/Z$. 
For each subset $Z \subseteq \{a,b,c\}$, let $H_Z$ be a graph representing $N/Z$. 
{Write $H_a$ for $H_{\{a\}}$, $H_{ab}$ for $H_{\{a,b\}}$, and so on.}
By Lemma \ref{B6}, $N/Z$ is vertically 3-connected and $H_Z$ is 2-connected and type 3. 
{Since $N/Z$ has rank more than five,}
by Theorem \ref{duck},  
$N/Z$ is uniquely represented by $H_Z$. 
In fact, $N/Z$ has rank at least seven, so $H_Z$ has at least seven vertices. Because $N/Z$ has no non-trivial series class, $H_Z$ has no deficient vertices, and so $H_Z \del e$ has at most two deficient vertices. Every graph in $\Ggg$ with at least seven vertices has at least three deficient vertices, so $H_Z \del e$ is not in $\Ggg$, and so $H_Z \del e$ cannot be obtained from $G/Z$ as a rearrangement. 
Thus for each $Z \subseteq \{a,b,c\}$, $G/Z = H_Z \del e$, up to replacement. 

For each edge $z \in \{a,b,c\}$ the contraction $G/z$ induces a map $V(G) \to V(G/z)$ in the obvious way: each vertex that is not an end of $z$ maps to itself, and each vertex that is an end of $z$ maps to the new vertex resulting from the identification of the ends of $z$ in the contraction operation. 
Let us denote the new vertex of $G/z$ resulting from the contraction of the edge $z$ by $\bar{u}^{z}$. 
Then for every vertex $x \in V(G)$, $x \mapsto \bar{u}^z$ if and only if $x$ is incident with $z$. 
For each $z \in \{a,b,c\}$, let us denote by $\tau_z \: V(G) \to V(G/z)$ this map defined by contracting $z$ in $G$. 
For each $z \in \{a,b,c\}$, $G/z = H_z \del e$
{up to replacement}. 
Let us denote by $V_z(e) = \{v_z, v_z'\}$ the set of vertices of $G/z$ that are the ends of $e$ in $H_z$ (where possibly $v_z = v_z'$ if $e$ is a loop). 
{(There are at most two internal vertices of lines in these graphs, so the possibility of replacements does not cause any issue here.)}
{When contracting an arbitrary subset $Z$ of $X$, we may compose these contraction maps in the obvious way to obtain a ``vertex contraction map" $\tau_Z \: V(G) \to V(G/Z)$. 
Again, write $\tau_a$ for $\tau_{\{a\}}$, $\tau_{ab}$ for $\tau_{\{a,b\}}$, and so on. 
Denote by $V_Z(e)$ the set of vertices of $G/Z$ that are ends of $e$ in $H_Z$.}
Then 
{$\tau_Z\inv(V_{Z}(e))$}
is a set of vertices of $G$; let us write 
{$\tau_Z\inv(e)$}
for short to denote the inverse image under 
{$\tau_Z$}
of the ends of $e$ in 
{$G/Z$}. 
Note that for each $z \in \{a,b,c\}$, $1 \leq |\tau_z\inv(e)| \leq 3$: 
if $e$ is a loop incident to a vertex $x \in V(H_z)$ that is not the new vertex $\bar{u}^z \in V(G/z)$ resulting from the contraction of $z$, then $\tau_z\inv(e) = \{x\}$. 
If $e$ has a pair of distinct ends $x,y \in V(H_z)$, neither of which are the new vertex $\bar{u}^z$, then $\tau_z\inv(e) = \{x,y\}$; if $e$ is a loop incident to $\bar{u}^z$ then $\tau_z\inv(e)$ is the pair of ends of $z$ in $G$. 
If $e$ has $\bar{u}^z$ as one end and a different vertex $x \in V(H_z)$ as its other end, then $\tau_z\inv(e) = \{v_1,v_2,x\}$, where $v_1, v_2$ are the ends of $z$ in $G$. 
{Observe further that $1 \leq |\tau_{abc}\inv(e)| \leq 5$, because $G/\{a,b,c\} = H_{abc} \del e$ up to replacement, and just two of the edges $a$, $b$, $c$ may share a common endpoint in $G$.
Moreover, because for each $Z \subseteq \{a,b,c\}$ the matroid $N/Z$ is uniquely represented by $H_Z$, we have  \[ H_a/\{b,c\} = H_b/\{a,c\} = H_c/\{a,b\} = H_{ab}/c = H_{ac}/b = H_{bc}/a = H_{abc}. \]
Thus, as $e$ has exactly one or two endpoints in each of these graphs, it directly follows that 
\[ 1 \leq | \tau_a\inv(e) \cap \tau_b\inv(e) \cap \tau_c\inv(e) | \leq 2.\]
If $\tau_a\inv(e) \cap \tau_b\inv(e) \cap \tau_c\inv(e) = \{x\}$, then let $H$ be the graph obtained by adding $e$ to $G$ as a loop incident to $x$. 
If $\tau_a\inv(e) \cap \tau_b\inv(e) \cap \tau_c\inv(e) = \{x,y\}$, then let $H$ be the graph obtained by adding $e$ as an edge to $G$ with endpoints $x$ and $y$. 
}

\begin{claim} 
$B(H)$ is our required twin for $N$. 
\end{claim}

The claim follows immediately from the fact that 
for each subset $Z \subseteq \{a,b,c\}$, $H/Z = H_Z$. 
\end{proof}

\subsection{Types 1 \& 2} 

We now show that if $N$ is an excluded minor for the class of bicircular matroids, and has rank at least 
{eight}, then $N$ has an element $e$ such that $N \del e$ is essentially 3-connected. 
We use the following version of Bixby's Lemma. 

\begin{thm}[Bixby's Lemma] \hypertarget{bixby} 
Let $e$ be an element of a vertically 3-connected matroid $M$. 
Then either $M \del e$ or $M/e$ is essentially 3-connected. 
\end{thm}

Let $N$ be an excluded minor for the class of bicircular matroids. 
Our goal is to find a triple of elements $X = \{a,b,c\}$ such that for every subset $Z \subseteq X$, $N/Z$ is vertically 3-connected. 
If $N$ has an element $e$ such that $N \del e$ is essentially 3-connected and type 3, then by Lemmas \ref{eggbert} and \ref{billbill} we are done. 
So suppose $N$ has no such element, and 
let $e \in E(N)$. 
Then either 
\begin{enumerate}
\item $N \del e$ is essentially 3-connected and type 1 or type 2, or 
\item $N \del e$ has an essential 2-separation. 
\end{enumerate}  
We show that the second case does not occur. 

\begin{lem}  \label{zevon} 
Let $N$ be an excluded minor for the class of bicircular matroids, of rank at least eight. 
Let $e \in E(N)$ and assume 
$N/e$ is essentially 3-connected. 
Then $N$ has an element $f$ such that $N \del f$ is essentially 3-connected. 
Moreover, if $N/e$ is type 3, then $N$ has an element $f$ such that $N \del f$ is essentially 3-connected and type 3. 
\end{lem}

\begin{proof} 
Let $G$ be a graph representing $N/e$. 
Since $N/e$ is essentially 3-connected with no non-trivial series class, $G$ is 2-connected with minimum degree at least 3. 
If $G$ has a 2-separation, then let $(A,B)$ and $(C,D)$ be proper 2-separations of $G$ with $A \subseteq C$ and $D \subseteq B$ with $A$ and $D$ minimal. 
By Proposition \ref{xavier}, there are edges $f \in A$ and $f' \in D$ such that $G \del \{f,f\}$ is 2-connected. 
If, on the other hand, $G$ is 3-connected, then choose an edge $f \in E(G)$. 
If $G \del f$ remains 3-connected, choose a second edge $f'$: $G \del \{f,f'\}$ is 2-connected. 
If $G \del f$ has a 2-separation, again applying Proposition \ref{xavier}, $G \del f$ has an edge $f'$ such that $G \del \{f,f'\}$ remains 2-connected. 
Thus $N /e \del \{f,f\}$ is essentially 3-connected. 

If $N \del \{f,f'\}$ is essentially 3-connected, then $N \del f$ is essentially 3-connected. 
Otherwise, as $N \del \{f,f'\} /e$ is essentially 3-connected, $N \del \{f,f'\}$ has a unique essential 2-separation $(A,B)$. 
If 
{each}
of $f$ and $f'$ were spanned by one side of the separation $(A,B)$, then $(A,B)$ would extend to a 2-separation of $N$. 
Since $N$ is vertically 3-connected, this does not occur. 
Thus one of $f$ or $f'$, without loss of generality let us say $f'$, is not spanned by $A$ nor by $B$. 
Therefore $N \del f$ is essentially 3-connected. 

Now if $N/e$ if type 1 or 2, or if $N/e$ is type 3 and also $N \del f$ is type 3, we are done. 
So assume that $N/e$ is type 3 but $N \del f$ is type 1 or type 2. 
Suppose first $N \del f$ has a type 1 representation $H$. 
{As $\Ggg$ contains no type 1 graphs, and $H/e$ remains type 1, $H/e \notin \Ggg$.} 
The graph $G \del f$ is 2-connected and has at most two vertices of degree 2. 
Let $v$ be the apex vertex of $H$. 
Since after possibly rolling and replacement, $H/e = G \del f$, $G \del f$ is type 1 with apex $v$, and $f$ does not have $v$ as an end in $G$ (else $G$ would already be type 1, but $G$ is type 3). 
The subgraph $(G-v) \del f$ has at least six vertices and at least two leaves, and at most two vertices that have degree two in $G \del f$. 
If $(G-v) \del f$ has just two leaves and both have degree 2 in $G \del f$, then each non-leaf vertex has an incident edge whose other end is $v$: let $f''$ be such an edge. 
Otherwise, $(G-v) \del f$ has a leaf $x$ whose degree in $G$ is at least 3, and so has at least two $x$-$v$ edges or an incident loop: let $f''$ be one of these $x$-$v$ edges or such a loop. 
In either case, $G \del f''$ is 2-connected and type 3. 
Let $W$ be the set of non-loop edges of $G$ other than $f''$; observe that $f''$ is contained in the closure of $W$ in $N/e$. 
We claim that $N \del f''$ is essentially 3-connected and type 3. 
For let $K$ be a graph representing $N \del f''$. 
Both $K/e$ and $G \del f''$ represent $N/e \del f''$; 
since $G \del f''$ has at least seven vertices and at most two deficient vertices, $G \del f''$ is not in $\Ggg$. 
Thus $K/e = G \del f''$ up to replacement, 
{so}
$K/e$ is type 3, and so $K$ is type 3 
(by Proposition \ref{sym29}). 
But suppose $N \del f''$ has an essential 2-separation $(A,B)$, say with $e \in A$. 
{Let $K/e$ have a compatible vertex labelling with $G \del f''$.} 
Since $G \del f''$ is 2-connected, $N \del f'' /e$ is essentially 3-connected. 
Thus by Proposition \ref{cam}, $K[A]$ consists of a pendant set of lines possibly together with a set of balloons, where $e$ is the single edge in a trivial line of this pendant set. 
Thus in $K/e$, the subgraph induced by $A-e$ consists of a set of balloons incident to a common vertex of attachment. 
As $K/e = G \del f''$ up to replacement, and $G \del f''$ is 2-connected, this implies that $A-e$ consists of a set of loops incident to a common vertex in $K/e$. 
Recall that $W$ is the set of non-loop edges of $G$ other than $f''$. So $W \subseteq B$ and, as $f'' \in \cl_{N/e}(W)$, also $f'' \in \cl_N(W)$. 
Thus $f'' \in \cl_N(B)$, which implies that $(A, B \cup f'')$ is an essential 2-separation of $N$. 
But this contradicts Lemma \ref{Nvert3}, so $N \del f''$ is essentially 3-connected. 

The case that $N \del f$ is type 2 is similar to the previous case. 
Let $H$ be a type 2 graph representation of $N \del f$, with apex $v$ and rotation vertex $u$. 
{Then $H/e$ has at least seven vertices and at most two deficient vertices, so cannot be in $\Ggg$. So}
{up} to rolling, rotation, and replacement, $H/e = G \del f$. 
{Let $H/e$ and $G \del f$ have compatible vertex labelings.}
Thus $G \del f$ is 
{either type 1 or}
type 2 with apex $v$ and
{(if type 2)}
rotation vertex $u$. 
The subgraph $(G-v) \del f$ has at least six vertices, and at least one vertex of degree 1. 
The graph $G \del f$ has at most two vertices of degree 2, with all remaining vertices of degree at least 3. 
Hence we either find a vertex different than $u$ of $(G-v) \del f$ that has at least two incident edges whose other end is $v$, or a vertex different than $u$ that is not pendant in $(G-v) \del f$ but has an incident edge whose other end is $v$. 
Choose such an edge $f'' \neq f$. 
Then $G \del f''$ is 2-connected and type 3, and $f''$ is in the closure in $N/e$ of the set of non-loop edges of $G \del f''$. 
The rest of the argument is similar to that in the previous case. 
Let $K$ be a graph representing $N \del f''$. 
{As before, $G \del f''$ is not in $\Ggg$, so}
$K/e = G \del f''$ up to replacement. 
{Thus}
$K/e$ is type 3, and so, so is $K$ type 3. 
If $N \del f''$ has an essential 2-separation $(A,B)$, with $e \in A$, then 
as in the previous case, we deduce that $f'' \in \cl_N(B)$ 
and so that $(A,B)$ extends to a vertical 2-separation of $N$, a contradiction. 
\end{proof}

To complete this section, we just need to show that we can find our desired triple of elements $a,b,c$, in the case that $N$ is an excluded minor of sufficient rank that does not have an element $e$ for which either $N \del e$ or $N/e$ is essentially 3-connected and type 3. 
So let $N$ be such an excluded minor. 
By Lemma \ref{zevon} and the discussion preceding, $N$ has an element $e$ for which $N \del e$ is essentially 3-connected and type 1 or 2. 

%
%
%
%

\begin{lem} \label{guitar} 
\addcontentsline{toc}{subsubsection}{Finding $a,b,c$ in case T1 \& T2}
Let $N$ be an excluded minor for the class of bicircular matroids. 
Assume $N$ has rank 
{at least ten}, and that $N$ does not have an element whose deletion or contraction is essentially 3-connected and type 3. 
Let $e$ be an element such that $N \del e$ is essentially 3-connected.  
Then there is a bicircular twin $M$ for $N$, relative to a set $X$ of size three, represented by a graph $H$ with $|V_H(X)| \geq 5$, 
and such that for every subset $Z \subseteq X$, $H/Z$ is 2-connected, and $N/Z$ is vertically 3-connected. 
\end{lem} 

\begin{proof} 
Let $G$ be a graph representing $N \del e$. 
Then $G$ is type 1 or type 2. 
By Proposition \ref{pedro}, $G$ has a unique apex $v$, and $v$ is not an internal vertex of a line or balloon of $G$. 
By Lemma \ref{deficientdan}, $N \del e$ has at most two non-trivial series classes, and by Proposition \ref{useful}, the non-trivial series classes of $N \del e$ correspond precisely to the lines and balloons of $G$. 

\bigskip 
\paragraph{Case 1.} \emph{$N$ has an element $e$ such that $N \del e$ is essentially 3-connected and type 2.} 

Assume $G$ is type 2. 
Let $u$ be the rotation vertex of $G$. 

\begin{claim} 
If $G$ is not 2-connected, then $G$ has an element $f$ such that $N \del f$ is essentially 3-connected and represented by a 2-connected type 2 graph. 
\end{claim}

\begin{proof}[Proof of Claim]
By replacement if necessary, we may assume each balloon of $G$ is standard. 
Since $G$ is type 2, every balloon of $G$ has vertex of attachment $v$. 
By Lemma \ref{deficientdan}, if $G$ has two balloons then each has size at most three, while if $G$ has just one balloon it may have size at most four, and if it has size four then $G$ does not have a 
line; moreover, $G$ has at most one line, which may have size at most two. 
Thus there are at least two vertices in the unique component $T$ of $G-\{u,v\}$ that is a tree, and in any case there is an edge $f \in G[V(T) \cup v]$, not contained in a line or balloon of $G$, such that $\co(G) \del f$ is 2-connected. 
Note that since $f$ is not incident to $u$, $G \del f$ remains type 2. 
Let $H$ be a graph representing $N \del f$. 
By Lemma \ref{B5}, $G \del f = H \del e$ up to rotation and replacement, $H$ is 2-connected, $N \del f$ is essentially 3-connected, no balloon of $G$ has size greater than three, and in $H$ the edge $e$ repairs the balloon or balloons of $G$. 
This
implies that $G$ has only one balloon
$S$, as otherwise $H$ would have $v$ as a cut vertex, and that $|S| \leq 3$. 

Further, this implies that $S$ has size two and that $S$ is a pair of parallel edges in $H$: 
for suppose to the contrary that the balloon $S$ of $G$ has size three or that
{the subgraph}
$H[S]$ consists of a pendant edge with a loop. 
If $|S|=3$,
{then}
$H[S]$ must form a standard balloon
{(else, by statement (ii) of Lemma \ref{B5}, in $H$ the edge $e$ must have as endpoints two internal vertices of the balloon $H[S]$; but then $H$ would not be 2-connected, contrary to statement (iii) Lemma \ref{B5})}.
{Thus whether $|S|=3$ or $H[S]$ consists of a single edge and a loop}, 
{because in $H$ the edge $e$ repairs the balloon $S$ of $G$, and by the 2-connectivity of $H$ edge $e$ does not have $v$ as an endpoint,}
$H$ is a type 3 representation of the essentially 3-connected matroid $N \del f$, contrary to assumption. 

Since $G \del f$ is type 2, so is $H \del e$ type 2; 
since $H[S]$ is a pair of parallel edges both incident to $v$, so also $H$ is type 2. 
\end{proof}

By the previous claim, we may now assume that $G$ is 2-connected. 

\begin{claim} 
There is a triple of edges $X = \{a,b,c\}$ in $G$, none of which are loops, 
spanning at least five vertices such that $G/Z$ is 2-connected and type 2, for every subset $Z \subseteq X$. 
\end{claim}

\begin{proof}[Proof of Claim]
Let $w$ be the non-rotation vertex incident to exactly two of the rotation lines (or edges) of $G$. 
Since $G$ is 2-connected and $G-\{u,v\}$ is a tree, if $G-u$ is not 2-connected, then $w$ has degree 1 in $G-u$. 
If this is the case, then either $G-\{u,w\}$ is 2-connected or there is a line $L$, by Lemma \ref{deficientdan} of length at most three, with ends $w$ and $x \in V(G-\{u,v\})$ such that the graph $G'$ obtained from $G$ by deleting $u$, $w$, and the internal vertices of $L$ is 2-connected. 
Let $a$ be the edge in $G-u$ incident with $w$, and consider the 2-connected graph $G'$
{defined in the previous sentence}
(where if $G-\{u,w\}$ is 2-connected then set $G' = G-\{u,w\}$). 
The graph $G'$ has at least four vertices and $G'-v$ is a tree; clearly we may choose a pair of non-adjacent edges $b,c$ of $G'$ such that $G'/\{b,c\}$ remains 2-connected. 
Let $X = \{a,b,c\}$. 
Then $X$ spans at least five edges in $G$, and $G/Z$ remains 2-connected for each $Z \subseteq X$. 

So now assume $G-u$ is 2-connected. 
Write $G'=G-u$. 
Let $S$ denote the set of contractible edges of $G'$. 
Assuming that $\si(G')$ is 
{neither}
$K_{2,n}$ nor $K_{2,n}'$, $G'[S]$ has a three-edge matching $X = \{a,b,c\}$, by Proposition \ref{marty}. 
By Proposition \ref{sax} $G'/Z$ is 2-connected for every subset $Z \subseteq X$. 
As long as none of $a, b, c$ is an edge linking $w$ and $v$, $G/Z$ also remains 2-connected and is type 2, for each $Z \subseteq X$. 
But suppose one of $a, b, c$ is a $v$-$w$ edge; without loss of generality, suppose $a = vw$. 
Since $G'$ is 2-connected there is
{an} 
edge $a' \in E(G')$ incident to $w$ whose other end is not $v$. 
Let $X' = \{a',b,c\}$. 
This edge $a'$ is contractible {in $G'$}, and unless the three edges $a',b,c$ induce a subgraph containing a $v$-$w$ path, $G/Z$ remains 2-connected for every $Z \subseteq X'$. 
Since none of $a'$, $b$, nor $c$ is incident to $v$, this does not happen. 
Hence $G/Z$ remains 2-connected and type 2 for all $Z \subseteq X'$. 

If $\si(G')$ is $K_{2,n}$ or $K_{2,n}'$, then choose a pair of edges $a,b \in E(G')$ non-incident to $v$ and an edge $c$ incident to $v$ but non-adjacent to $a$ and non-adjacent to $b$. 
Since $|V(G')| \geq 7$, this is clearly possible. 
Set $X = \{a,b,c\}$. 
Clearly $G/Z$ remains 2-connected and type 2 for each $Z \subseteq X$. 
\end{proof}

Let $X$ be a set of three edges as given by the previous claim. 
For each subset $Z$ of $X$, let $H_Z$ be a graph representing $N/Z$. 

\begin{claim} 
For every subset $Z$ of $X$, $H_Z$ is 2-connected and type 2 with apex $v$, $e$ has $v$ as an end, and $N/Z$ is vertically 3-connected.
\end{claim}

\begin{proof}[Proof of Claim]
Since $G/Z$ is 2-connected, $N \del e /Z$ is essentially 3-connected. 
Each of $G/Z$ and $H_Z \del e$ represent $N \del e /Z$. Because $N$ has rank at least ten, $G/Z$ has at least seven vertices. By Lemma \ref{junniper}, $G/Z$ has at most two deficient vertices, so $G/Z$ cannot be in $\Ggg$. Thus 
by Lemma \ref{duckduck}, $G/Z = H_Z \del e$ up to rotation and replacement.
Thus $H_Z \del e$ is type 2 with apex $v$ and is 2-connected up to replacement. 
But no replacement operation applied to a 2-connected graph yields a graph with a 1-separation, so $H_Z \del e$ is 2-connected. 

Suppose for a contradiction that $H_Z$ has a proper 1-separation $(A,B)$, say with $V(A) \cap V(B) = \{x\}$ and $e \in A$. 
Then neither $A$ nor $B$ is a parallel class of $N/Z$, and 
since $N/Z$ has no non-trivial series class, $H_Z$ has neither a balloon nor a line, so 
$(A,B)$ is an essential 2-separation of $N/Z$. 
Since $N/Z \del e$ is essentially 3-connected, by Proposition \ref{cam}, $A-e$ is a balloon of $H_Z \del e$, with vertex of attachment $x$. 
But this contradicts the fact that $H_Z \del e$ is 2-connected. 
Because $H_Z$ is 2-connected, $N/Z$ is essentially 3-connected. 
Since $N/Z$ has no non-trivial series class, $N/Z$ is vertically 3-connected. 

Now suppose for a contradiction that $H_Z$ is not type 2. 
That is, $N/Z$ is type 1 or type 3. 
We have already seen that $H_Z \del e$ is type 2, so $H_Z$ cannot be type 1. 
So $N/Z$ is type 3. 
But $N/Z$ is a minor of one of $N/a$, $N/b$, or $N/c$, so by Proposition \ref{sym29}, one of these is type 3. 
Each of $N/a$, $N/b$ and $N/c$ is essentially 3-connected, so this is contrary to assumption. 
So $H_Z$ is type 2. 
Since $H_Z \del e$ is type 2 with apex $v$, $H_Z$ must also have apex $v$. 
Thus were $e$ not incident to $v$ in $H_Z$, $H_Z$ would be type 3. 
\end{proof}

\paragraph{Case 2.} \emph{$N \del e$ is type 1.} 

We first show that we may now further assume that $N$ has no element $e$ such that $N \del e$ or $N/e$ is essentially 3-connected and type 2. 

\begin{claim} \label{corelli} 
If $N$ has an element $e$ for which $N/e$ is essentially 3-connected and type 2, then $N$ has an element $f$ for which $N \del f$ is essentially 3-connected and represented by a 2-connected type 2 graph. 
\end{claim}

\begin{proof}
Let $G$ be a type 2 graph representing $N/e$, say with apex $v$ and rotation vertex $u$. 
Since $N/e$ has no non-trivial series class, $G$ is 2-connected with minimum degree at least three. 
The graph $G-\{u,v\}$ is a tree with at least one leaf $x$ that is not incident to $u$. 
Since $\deg_G(x)$ is at least three, there are at least two $v$-$x$ edges in $G$. 
Let $f$ be a $v$-$x$ edge. 
Then $G \del f$ is 2-connected, and $N/e \del f$ is essentially 3-connected. 
Let $H$ be a graph representing $N \del f$. 
By Lemma \ref{duckduck}, $H/e = G \del f$, up to 
{rotation, replacement, and rearrangement. But $G \del f$ has at least nine vertices and at most two vertices of degree two, so $G \del f$ is not in $\Ggg$. Thus $H/e = G \del f$ up to rotation and replacement.}

We claim $N \del f$ is essentially 3-connected. 
For suppose to the contrary that $N \del f$ has an essential 2-separation $(A,B)$, say with $e \in A$. 
Since $N \del f /e$ is essentially 3-connected, by Proposition \ref{cam}, $H[A]$ consists of a pendant set of lines together possibly with a set of balloons. 
Contracting $e$ in $H$ thus yields at least one balloon incident to the single vertex in $V(A) \cap V(B)$. 
Since $N \del f /e$ is type 2, represented by $G \del f$, this implies $V(A) \cap V(B) = \{v\}$. 
But it is clear that $f$ is in $\cl_N(B)$, so the 2-separation $(A,B)$ of $N \del f$ extends to a vertical 2-separation $(A, B \cup f)$ of $N$, contrary to Lemma \ref{Nvert3}. 
Thus $H$ is 2-connected and $N \del f$ is essentially 3-connected. 

By assumption, $N$ has no element whose deletion or contraction is essentially 3-connected and type 3. 
Thus $N \del f$, and $H$, is either type 1 or type 2. 
But $H/e = G \del f$ up to rotation and replacement, and $G \del f$ is type 2, so $H/e$ is type 2. 
If $H$ is type 1, then $H/e$ cannot be type 2. 
Thus $H$ is type 2. 
\end{proof} 

By the previous section and the previous claim, we may now assume that $N$ does not have an element $e$ for which $N \del e$ or $N/e$ is essentially 3-connected and type 2 or type 3. 

{Returning now to our main argument: recall that $G$ is a graph representing $N \del e$ where $N \del e$ is essentially 3-connected, $G$ is type 1 with apex vertex $v$, and since $N \del e$ has rank at least ten, $G$ has at least ten vertices.}
Let us assume that $G$ is substandard; 
that is, every balloon of $G$ has $v$ as its vertex of attachment and every loop not properly contained in a balloon is incident to $v$. 

\begin{claim} 
If $G$ is not 2-connected, then $G$ has an element $f$ such that $N \del f$ is essentially 3-connected and represented by a 2-connected graph, up to rolling. 
\end{claim}

\begin{proof}[Proof of Claim]
Suppose $G$ is not 2-connected. 
Since $G$ is substandard, this implies that $G$ has a 1-separation $(A,B)$ with $V(A) \cap V(B) = \{v\}$, in which $G[A]$ consists of one or two balloons and $G[B]-v$ is a tree $T$. 
The component $T$ of $G-v$ has at least two leaves, at most one of which may be an internal vertex of a line, by Lemma \ref{deficientdan}. 
Let $u$ be a leaf of $T$ that is not an internal vertex of a line. 
Then there are at least two $v$-$u$ edges in $G$. 
Let $f$ be a $v$-$u$ edge. 
Then $f$ is not contained in a line or balloon of $G$ and $\co(G) \del f$ remains 2-connected. 
Let $H$ be a graph representing $N \del f$. 
By Lemma \ref{B5}, $N \del f$ is essentially 3-connected, and $H$ is 2-connected up to rolling. 
\end{proof} 

By the claim, we may now assume that $G$ is 2-connected. 

\begin{claim} 
There is a triple of edges $X = \{a,b,c\}$ in $G$ spanning at least five vertices, where none of $a$, $b$, nor $c$ is a loop,
such that $G/Z$ is 2-connected, for every subset $Z \subseteq X$. 
{Moreover, we may choose $a$, $b$, and $c$ so that whenever $G$ has a non-trivial line $L$ with $v$ as one of its vertices of attachment, then each of $a$, $b$, $c$ is in the closure of $E(G)-L$ in $N \del e$.}
\end{claim}

\begin{proof}[Proof of Claim]
{Suppose first $G$ does not have a non-trivial line with $v$ as one of its vertices of attachment.}
Let $S$ denote the set of contractible edges of $G$. 
Assuming that $\si(G)$ is not $K_{2,n}$ nor $K_{2,n}'$, $G[S]$ has a three-edge matching $X = \{a,b,c\}$, by Proposition \ref{marty}. 
By Proposition \ref{sax} $G/Z$ is 2-connected for every subset $Z \subseteq X$. 

So now assume that $\si(G)$ is $K_{2,n}$ or $K_{2,n}'$. 
Then the apex vertex $v$ of $G$ is one of the two vertices of $\si(G)$ of degree $n$ or $n+1$; let $u$ be the other vertex of degree $n$ or $n+1$ in $\si(G)$, and let $Y$ be the set of remaining vertices of $G$. 
Let $a=uy_1$ and $b=uy_2$ be a pair of vertices sharing $u$ as one end and with distinct ends $y_1 \neq y_2$ in $Y$, and let $c$ be an edge with ends $v$ and $y_3 \in Y$, with $y_3 \notin \{y_1,y_2\}$. 
As {$|Y| \geq 8$}, this is clearly possible. 
It is also clear that $G/Z$ remains 2-connected and type 1 for every subset $Z \subseteq X$. 

{Now suppose $G$ has a non-trivial line with $v$ as one of its vertices of attachment. By Proposition \ref{deficientdan}, either $G$ has just one line with at most 4 edges, or $G$ has two lines, one with at most 3 edges, the other with 2 edges. Let $G'$ be the graph obtained by removing the edges and internal vertices of the line or lines of $G$. Then $G'$ remains 2-connected, and has at least 7 vertices. Thus by Proposition \ref{marty}, unless $\si(G')$ is $K_{2,n}$ or $K_{2,n}'$, $G'[S]$ has a three-edge matching $X = \{a,b,c\}$, where $S$ is the set of contractible edges of $G'$, and by Proposition \ref{sax} $G'/Z$ is 2-connected for every subset $Z \subseteq X$. Thus also $G/Z$ is 2-connected up to rolling. 
If $\si(G')$ is $K_{2,n}$ or $K_{2,n}'$, then 
as before let $Y$ be the set of vertices of $G'$ other than the two vertices of degree $n$ or $n+1$ in $\si(G')$. Then $|Y| \geq 5$, so we may choose three edges $a, b, c$ just as in the case of the previous paragraph in which $G$ has no non-trivial line with $v$ as one end while $\si(G)$ is $K_{2,n}$ or $K_{2,n}'$. Observe that just as in the case $\si(G)$ is neither $K_{2,n}$ nor $K_{2,n}'$, $G'$ remains 2-connected. Thus in either case, $B(G')$ is connected. This implies that for each non-trivial line $L$ with $v$ as one of its ends, each of $a$, $b$, and $c$ are in the closure of $E(G) - L$ in $N \del e$. }

{Finally, observe that in any case none of $a$, $b$, nor $c$ is a loop.}
\end{proof} 

{Let $X$ be a set of three edges $a$, $b$, $c$, as guaranteed to exist by the previous claim.}
For each subset $Z$ of $X$, let $H_Z$ be a graph representing $N/Z$. 

\begin{claim} 
For every subset $Z$ of $X$, $H_Z$ is 2-connected and type 1 with apex $v$, up to rolling $e$ has $v$ as an end, and $N/Z$ is vertically 3-connected. 
\end{claim}

\begin{proof}[Proof of Claim]
Since $G/Z$ is 2-connected, $N \del e /Z$ is essentially 3-connected. 
Each of $G/Z$ and $H_Z \del e$ represent $N \del e /Z$; 
by Lemma \ref{duckduck}, $G/Z = H_Z \del e$
{up to rolling, rotation, replacement, and rearrangement. But $G/Z$ is type 1, and no type 1 graphs are in $\Ggg$, so neither rotation nor rearrangement are possible. Thus $G/Z = H_Z \del e$}
up to rolling and replacement. 
Thus $H_Z \del e$ is type 1 with apex $v$ and 2-connected up to rolling. 

Now suppose for a contradiction that $H_Z$ has a proper 1-separation $(A,B)$, say with $V(A) \cap V(B) = \{x\}$ and $e \in A$. 
Then neither $A$ nor $B$ is a parallel class of $N/Z$. 
Since $N/Z$ has no non-trivial series class, $H_Z$ has no balloon and no line, and 
$(A,B)$ is an essential 2-separation of $N/Z$. 
Since $N/Z \del e$ is essentially 3-connected, by Proposition \ref{cam}, $A-e$ is a balloon in $H_Z \del e$, with vertex of attachment $x$. 
This implies $e \in \cl_{N/Z}(A-e)$. 
Since $G/Z$ is 2-connected up to rolling, $x \neq v$ (that is, the balloon $A-e$ in $H_Z \del e$ does not have $v$ as its vertex of attachment). 
Thus $A-e$ is a line of $G/Z$ with vertices of attachment $x$ and $v$. 
{Since the contraction of $Z$ in $G$ cannot produce a line not already present in $G$}, $A-e$ is a line of $G$. 
{Now, $H_Z$ is a minor of one of $H_a$, $H_b$, or $H_c$, let us say, $H_a$. If $x$ were not already a cut vertex in $H_a$, then, because $H_a \del e = G/a$ up to rolling and replacement, by our choice of $a, b, c$, $x$ would not be a cut vertex in $H_Z$, a contradiction. Thus the proper 1-separation $(A,B)$ of $H_Z$ extends to a proper 1-separation $(A,B')$ of $H_a$, where $B \subseteq B'$. 
This separation $(A,B')$ is a vertical 2-separation of $N/a$. 
By our choice of $a$, $a \in \cl_{N \del e}(B')$, so $a \in \cl_N(B')$. But this implies that $(A,B')$ extends to a vertical 2-separation $(A,B' \cup \{a\})$ of $N$, contrary to Lemma \ref{Nvert3}. Thus $H_Z$ is 2-connected.}

Because $H_Z$ is 2-connected, $N/Z$ is essentially 3-connected. 
Since $N/Z$ has no non-trivial series class, $N/Z$ is vertically 3-connected. 

Now suppose for a contradiction that $H_Z$ is not type 1. 
That is, $N/Z$ is type 2 or type 3. 
But this is impossible, since by assumption $N$ does not have an element $f$ for which $N/f$ is essentially 3-connected and type 2 or type 3. 
Since $N/z$ is type 1 for each $z \in \{a,b,c\}$, $N/Z$ is also type 1. 
Since $H_Z \del e$ is type 1 with unique apex $v$, $H_Z$ must also have apex $v$. 
If in $H_Z$ edge $e$ had two distinct ends neither of which were $v$, then $H_Z$ would be type 3, a contradiction. 
\end{proof} 

\paragraph{\emph{Constructing a bicircular twin for $N$.}}
We now have, in each of cases 1 and 2 above, a 2-connected graph $G$ representing $N \del e$, a set $X = \{a,b,c\} \subseteq E(G)$, 
{no element of which is a loop},
such that for each subset $Z$ of $X$, $N/Z$ is vertically 3-connected. 
In the case that $G$ is type 2, $N/Z$ is type 2 for each $Z \subseteq X$, and when $G$ is type 1, so is $N/Z$. 

We now construct a bicircular twin for $N$ relative to $X = \{a,b,c\}$. 
Just as in the proof of Lemma \ref{billbill}, 
we use the ``contraction maps"
{$V(G) \to V(G/Z)$, for each $Z \subseteq X$, defined by composing single-element contraction maps $V(G) \to V(G/z)$ and $V(G/z) \to V((G/z)/z')$, for $z, z' \in X$, and so on; each vertex maps either to itself or, if it is an endpoint of the contracted edge $z$, to the new vertex $\bar{u}^z$ resulting from the contraction}.
For each 
{$Z \subseteq X$, by Lemma \ref{duckduck}, $G/Z = H_Z \del e$}
{up to rolling, rotation, replacement, and rearrangement. Because the set $\Ggg$ contains no type 1 graphs, and in the case $N/Z$ is type 2, $H_Z \del e$ has at least seven vertices, at most two of which have degree two, $H_Z \del e$ is not in $\Ggg$. Thus $G/Z = H_Z \del e$ up to rolling, rotation, and replacement.}

{As in the proof of Lemma \ref{billbill}, denote by $\tau_a$, $\tau_b$, $\tau_{ab}$, and so on, the contraction maps determined by contracting $a$, $b$, $\{a,b\}$, and so on, in $G$; 
and write $\tau_z\inv(e)$ for $z \in X$, and more generally, $\tau_Z\inv(e)$ for $Z \subseteq X$, to denote the set of vertices in $G$ that are the inverse image under $\tau_z$ (resp.\ $\tau_Z$) of the vertices in $G/z$ (resp.\ $G/Z$) that are the endpoints of $e$ in $H_z$ (resp.\ $H_Z$). 
Since $N$ does not have an element whose deletion or contraction is essentially 3-connected and type 3, for each $z \in X$, either}
in $H_z$ the edge $e$ has the apex $v$ as one end {or $H_z$ is type 1 and $e$ is a loop to which a rolling operation may be applied to obtain a substandard representation in which $e$ has $v$ as an endpoint. 
Thus without loss of generality we may assume, for each $z \in X$, that if $H_z$ is type 1 then $H_z$ is substandard. Thus for each $z \in X$, $\tau_z\inv(e)$ contains $v$. 
For each $Z \subseteq \{a,b,c\}$, the matroid $N/Z$ is vertically 3-connected with rank at least seven, so by Theorem \ref{duck} $N/Z$ is uniquely represented by $H_Z$ up to rolling and rotation. 
As in the proof of Lemma \ref{billbill}, the facts that, up to rolling and rotation, 
\[ H_a/\{b,c\} = H_b/\{a,c\} = H_c/\{a,b\} = H_{ab}/c = H_{ac}/b = H_{bc}/a = H_{abc} \] 
and that $e$ has exactly one or two endpoints in each of these graphs, immediately imply that 
\[ 1 \leq | \tau_a\inv(e) \cap \tau_b\inv(e) \cap \tau_c\inv(e) | \leq 2 \]
where the intersection whose size is bounded by this pair of inequalities contains $v$. 
If the size of this intersection is 1, then let $H$ be the graph obtained from $G$ by adding $e$ as a loop incident with $v$. 
Otherwise, when $\tau_a\inv(e) \cap \tau_b\inv(e) \cap \tau_c\inv(e) = \{v,x\}$, let $H$ be the graph obtained from $G$ by adding $e$ as an edge linking $v$ and $x$. 
}

\begin{claim} 
$B(H)$ is a twin for $N$ relative to $X$. 
\end{claim}

\begin{proof}[Proof of Claim]
For each $z \in \{a,b,c\}$, $H/z = H_z$. 
Thus for each $Z \subseteq \{a,b,c\}$, $H/Z = H_Z$, and so $B(H/Z) = N/Z$. 
\end{proof} 

This completes the proof of Lemma \ref{guitar}. 
\end{proof}

\section{No excluded minor has rank greater than nine} \label{wonderwoman} 

\subsection{Quasi-graphic matroids} 

For working with quasi-graphic matroids and finding minors, it is convenient to have additional information about the framework of a quasi-graphic matroid. 
For this we use the results of~\cite{MR4037634}. 
Let $M$ be a quasi-graphic matroid, and let $G$ be a framework for $M$. 
Setting $\Bb = \{C : C$ is a cycle of $G$ and a circuit of $M \}$ we obtain a biased graph $(G,\Bb)$ from the framework $G$. 
The collection $\Bb$ tells us which cycles of $G$ are circuits of $M$, but this does not determine $M$. 
A \emph{bracelet} in a biased graph is a pair of vertex-disjoint unbalanced cycles. 
In general, a bracelet in $(G,\Bb)$ may or may not be a circuit of $M$. 

A \emph{bracelet function} is a function from the set of bracelets of a biased graph $(G,\Bb)$ to the set $\{\mathsf{independent},\mathsf{dependent}\}$. 
Given a biased graph $(G,\Bb)$, the \emph{bracelet graph} $\mathscr B(G,\Bb)$ of $(G,\Bb)$ is the graph with vertex set the collection of bracelets of $(G,\Bb)$ in which 
two bracelets are joined by an edge if and only if their union has 
the property that the minimum number of edges that must be removed to obtain an acyclic subgraph is exactly three.
If $\chi$ is a bracelet function with the property that $\chi(Y)=\chi(Y')$ whenever $Y$ and $Y'$ are bracelets in the same component of $\mathscr B(G,\Bb)$, then $\chi$ is a \emph{proper} bracelet function.
Given a biased graph $(G,\Bb)$ with bracelet function $\chi$, let $\Cc(G, \Bb,\chi)$ be the collection of edge sets of: balanced cycles, thetas with no cycle in $\Bb$, tight handcuffs, bracelets $Y$ with
$\chi(Y)=\mathsf{dependent}$, and loose handcuffs containing bracelets $Y$ with $\chi(Y)=\mathsf{independent}$. 
The following two theorems are proved in~\cite{MR4037634}. 

\begin{thm}[\cite{MR4037634}, Theorem 2.1] \label{BraceletFunctionPropriety}
Let $(G,\Bb)$ be a biased graph with $G$ connected, 
and let $\chi$ be a bracelet function for $(G,\Bb)$. 
If $\Cc(G,\Bb,\chi)$ is the set of circuits of a matroid, then $\chi$ is proper.
\end{thm}

\begin{thm}[\cite{MR4037634}, Theorem 1.1] \label{equivalences}
Let $M$ be a matroid and let $(G, \Bb)$ be a biased graph with $E(G)=E(M)$. 
The following are equivalent. 
\begin{enumerate}[label={\upshape (\roman*)}]
\item There is a proper bracelet function $\chi$ for $G$ such that $M=M(G,\Bb,\chi)$. 
\item $M$ is quasi-graphic with framework $G$ and $\Bb$ is the set of cycles of $G$ that are circuits of $M$. 
\end{enumerate} 
\end{thm}

We say the quasi-graphic matroid $M = M(G,\Bb,\chi)$ is \emph{represented} by the triple $(G,\Bb,\chi)$ consisting of the biased graph $(G,\Bb)$ with bracelet function $\chi$. 

It is shown in~\cite{JGT:JGT22177} that if $M$ is quasi-graphic with framework $G$, then $M \del e$ and $M/e$ are quasi-graphic with framework $G \del e$ and $G/e$, respectively, where if $e$ is an unbalanced loop then $G/e$ is the graph $G^{\circ e}$ defined in Section \ref{minorsofbicircularmatroids} for contracting an unbalanced loop in a biased graph. 
It is shown in~\cite{MR4037634} that if $M$ is represented by $(G,\Bb,\chi)$, then 
$M(G,\Bb,\chi) \del e$ and $M(G,\Bb,\chi)/e$ are represented by naturally defined biased graphs $(G,\Bb) \del e$, $(G,\Bb)/e$ along with their inherited bracelet functions. 
For details, see Sections 2 and 4.2 of~\cite{MR4037634}.

\subsection{Six small excluded minors} \label{littleguys} 

In the following section, we show that there is no excluded minor of rank greater than nine by showing that any such purported matroid in fact already contains one of six smaller excluded minors. 
In this section we describe these matroids. 
Each is quasi-graphic, and so has a biased graphic representation. 
Some have more than one such representation; 
here we exhibit those representations that are useful for finding these matroids as minors in the proof of Theorem \ref{nomore}. 
Geometric representations and frameworks are shown in Figures \ref{LittleGuys1}-\ref{LittleGuys5}. 
Points in geometric representations are shown as solid discs, while vertices of graphs are open circles. 
The M-numbers labelling the matroids are their unique identifiers in the dataset described in~\cite{MR2389607}. 

\begin{figure}[tbp] 
\begin{center} 
\includegraphics[scale=0.95]{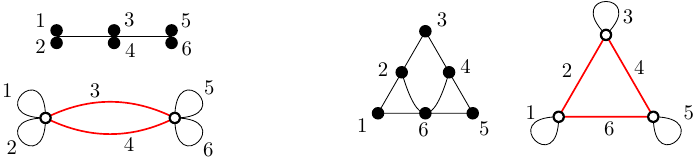}
\end{center} 
\caption{$M(2C_3)$ and $M(K_4)$}
\label{LittleGuys1} 
\end{figure} 
The graph $2C_3$ is obtained by adding an edge in parallel to each of the edges of a 3-cycle. 
Its cycle matroid $M(2C_3)$ has rank two and consists of three 2-element rank-1 flats, one of which is a balanced 2-cycle in the framework shown in Figure \ref{LittleGuys1}. 
At right in Figure \ref{LittleGuys1} is a framework for $M(K_4)$, with balanced cycle $\{2,4,6\}$. 
The matroid \emph{Pf} is shown in Figure \ref{LittleGuys2}, along with two of its frameworks (it is its framework at right that we find in a purported excluded minor). 
The framework shown at centre in Figure \ref{LittleGuys2} has a single balanced cycle $\{1,4,5,6\}$, and the framework at right 
has dependent bracelet $\{2,3\} \cup \{1,4\}$. 
\begin{figure}[tbp] 
\begin{center} 
\includegraphics[scale=0.95]{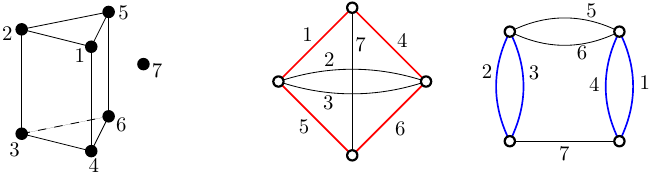}
\end{center} 
\caption{ \emph{Pf}}
\label{LittleGuys2} 
\end{figure} 
The framework for $\overline{P}$ shown in Figure \ref{LittleGuys3} has a single balanced cycle $\{1,2,3,4\}$. 
\begin{figure}[tbp] 
\begin{center} 
\includegraphics[scale=0.95]{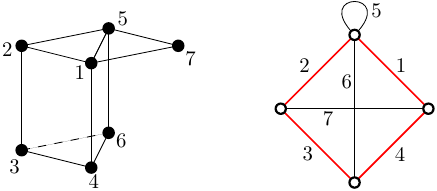}
\end{center} 
\caption{$\overline{P}$}
\label{LittleGuys3} 
\end{figure} 
The framework for $\overline{\underline{P}}$ shown in Figure \ref{LittleGuys4} has the single balanced cycle $\{3,4,5,6\}$. 
\begin{figure}[tbp] 
\begin{center} 
\includegraphics[scale=0.95]{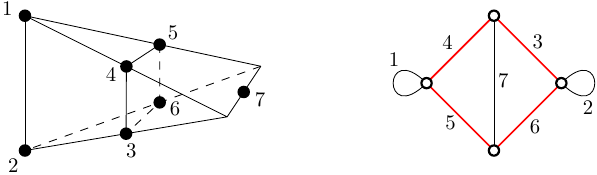}
\end{center} 
\caption{$\overline{\underline{P}}$}
\label{LittleGuys4} 
\end{figure} 
The matroid M2077 is the dual of the matroid whose geometric representation is shown in Figure \ref{LittleGuys5}. 
Its framework has one balanced cycle, $\{1,2,6,7,8\}$. 
\begin{figure}[tbp] 
\begin{center} 
\includegraphics[scale=0.95]{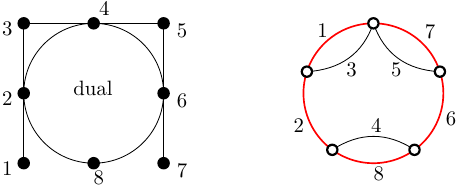}
\end{center} 
\caption{M2077 has 5-circuits 45678, 12678, 12348, and 4-circuit 1357.}
\label{LittleGuys5} 
\end{figure} 
{(The M-number for a matroid is its unique identifier in the dataset described in~\cite{MR2389607}.)}

It is straightforward (if labourious) to check that each of these matroids is minor-minimally non-bicircular. 
{We have verified that each is an excluded minor by computer search}
{(see \hyperref[appendix]{Appendix}).}
{As examples of how such verification may be done ``by hand", we provide the following simple checks that $M(2C_3)$ and $M(K_4)$ are indeed minor-minimally non-bicircular. }

{The matroid $M(2C_3)$ is the rank-2 matroid consisting of three distinct parallel classes each consisting two elements. 
A parallel class of elements in a bicircular graph representation may only be represented as a set of loops incident to a common vertex, and elements of distinct parallel classes must be incident to distinct vertices. 
A graph representation of a rank-2 matroid has exactly two vertices. 
Since $M(2C_3)$ has three distinct parallel classes, it cannot be represented by a graph with just two vertices, so $M(2C_3)$ is not bicircular. 
}

{Because $M(2C_3)$ has transitive automorphism group, we just need check that each of $M(2C_3) \del e$ and $M(2C_3)/e$ are bicircular, for an arbitrary element $e$ of its ground set. 
So let $e \in E(2C_3)$. 
The matroid $M(2C_3) \del e$ is the rank-2 matroid consisting of two distinct parallel classes of size two and a third element placed freely on the line spanned by them. 
The matroid $M(2C_3)/e$ is the rank-1 matroid consisting of a parallel class of four elements and one loop. 
Bicircular graph representations for each are shown in Figure \ref{M2C3isanexminor}. 
\begin{figure}[tbp] 
\begin{center} 
\includegraphics[scale=0.95]{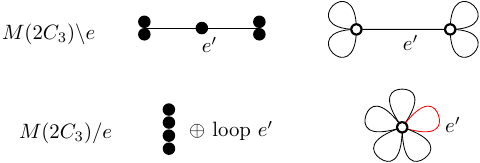}
\end{center} 
\caption{{Bicircular graph representations for $M(2C_3) \del e$ and $M(2C_3)/e$, where $\{e,e'\}$ is a parallel pair in $M(2C_3)$; $e'$ is a loop of $M(2C_3)/e$ and a balanced loop in its graph representation.}}
\label{M2C3isanexminor} 
\end{figure} 
}

{We now show that $M(K_4)$ is not bicircular. 
Each element of $M(K_4)$ is in exactly two 3-circuits, each of which is a rank-2 flat. 
Observe that a rank-2 flat in a bicircular representation consists precisely of the set of all edges incident to a pair of vertices. 
Thus if an element $e$ is contained in two distinct rank-2 flats $L, L'$ of a bicircular matroid, then in any bicircular graph representation $e$ is a loop (where, if $L$ consists of the set of all edges incident to vertices $u, v$, then $L'$ is the set of all edges incident to a pair of vertices $u, w$ for some vertex $w \notin \{u,v\}$, and $e$ is a loop with endpoint $u$). 
Toward a contradiction, suppose $G$ is a graph providing a bicircular representation for $M(K_4)$. The matroid $M(K_4)$ has rank 3, so $G$ has three vertices. 
Because each of its six elements is in exactly two rank-2 flats, each is represented by a loop in $G$. 
Since $M(K_4)$ has no two elements in parallel, no vertex of $G$ has more than one incident loop. 
But $G$ only has three vertices, so this is impossible. 
}

{The automorphism group of $M(K_4)$ is transitive, so we just need to check that each of $M(K_4) \del e$ and $M(K_4)/e$ is bicircular, for an arbitrary element $e \in E(M(K_4))$. 
Bicircular graph representations for each are shown in Figure \ref{MK4isanexminor}. 
\begin{figure}[tbp] 
\begin{center} 
\includegraphics[scale=0.95]{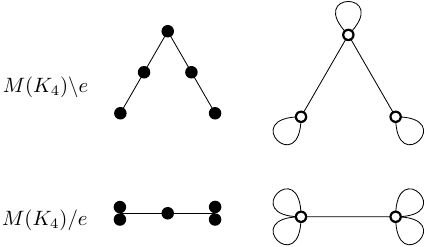}
\end{center} 
\caption{Bicircular graph representations for $M(K_4) \del e$ and $M(K_4)/e$.}
\label{MK4isanexminor} 
\end{figure} 
}

\subsection{No excluded minor has rank greater than nine} \label{warthog}

Let $N$ and $M$ be matroids on a common ground set $E$.  
We say that a set $Y \subseteq E$ is a \emph{disagreement set} if $r_N(Y) \neq r_{M}(Y)$. 

\begin{prop} \label{bob} 
If $Y$ is a minimal disagreement set for $N$ and $M$, then $Y$ is a circuit in one of $N$ or $M$ and an independent set in the other. 
\end{prop}

\begin{proof} 
Since the ranks of $N$ and $M$ disagree on $Y$, we may choose $C \subseteq Y$ so that $C$ is a circuit of
exactly one of $N$ or $M$; let us say $C$  is a circuit of $N$.  
If $C$ is independent in $M$, then $C$ is a disagreement set. 
Since every subset of $C$ is independent in $N$ and $M$, $C=Y$ and we are done.  
Otherwise, $C$ is dependent in $M$, and so there is a subset $C'$ of $C$ so that $C'$ is a circuit of $M$. 
Since $C$ is not a circuit in both $N$ and $M$, $C'$ is a proper subset of $C$. 
Thus $C'$ is independent in $N$, and so $C'$ is a disagreement set properly contained in $Y$, a contradiction.
\end{proof}

\begin{prop} \label{Hank} 
Let $N$ and $M$ be matroids on a common ground set $E$. 
Assume $N$ and $M$ are twins relative to $X \subseteq E$, and that $Y \subseteq E$ is a circuit in $N$ but independent in $M$. 
\begin{enumerate}[label={\upshape (\roman*)}]
\item No element of $X$ is in the span of $Y$ in $N$. 
\item Every element of $X$ is in the span of $Y$ in $M$.  
\item $Y$ is a circuit in $M/e$ whenever $e \in X$. 
\end{enumerate} 
\end{prop}

\begin{proof} 
Let $e \in X$. 
Suppose, contrary to (i), that $e \in \cl_N(Y)$. 
Then $r_{M/e}(Y) = r_{N/e}(Y) = r_N(Y)-1$. 
But $r_N(Y) = r_{M}(Y)-1$, so this implies that $r_{M}(Y)-2 = r_{M/e}(Y)$, a contradiction. 

Now suppose, contrary to (ii), that $e \notin \cl_M(Y)$. 
Then $r_{N/e}(Y) = r_{M/e}(Y) = r_M(Y)$. 
But $r_M(Y) = r_N(Y)+1$, so this implies that $r_{N/e}(Y) = r_N(Y)+1$, a contradiction. 

Statement (iii) follows from statement (i): 
since $e$ is not in $\cl_N(Y)$, $Y$ remains a circuit in $N/e = M/e$. 
\end{proof}

Note that statement (i) of Proposition \ref{Hank} implies that no minimal disagreement set for $N$ and $M$ contains an element of $X$. 

\begin{thm} \label{nomore} 
Let $N$ and $M$ be matroids on a common ground set $E$. 
Let $X =\{a,b,c\} \subseteq E$ be a subset of three elements such that 
\begin{enumerate}[label={\upshape (\roman*)}]
\item $N$ and $M$ are twins relative to $X$, 
\item For every subset $Z$ of $X$, $N/Z = M/Z$ is vertically 3-connected, and 
\item $M$ is bicircular, represented by a graph $G$ in which $|V_G(X)| \geq 5$. 
\end{enumerate} 
Then $N$ 
{contains} one of 
{$M(2C_3)$}, $M(K_4)$, 
$\overline{P}$, {Pf}, $\overline{\underline{P}}$, or M2077 as a minor. 
\end{thm}

\begin{proof} 
We proceed via a series of three claims.  
The first two will establish that $N$ is quasi-graphic. 
The third claim is that $N$ either contains  \emph{Pf} or is frame.  
Then from the graph representation of $M$ we obtain a biased graph representing $N$ as a frame matroid. 
Finally, we show that this biased graph contains a biased graph representation of one of the matroids listed as a minor. 

\begin{claimnumbered} \label{click} 
A minimal disagreement set is a circuit in $N$ and independent in $M$. 
\end{claimnumbered}

\begin{proof}[Proof of Claim]
Suppose to the contrary that a minimal disagreement set $Y$ is a circuit in $M$ and independent in $N$.  
Then $G[Y]$ forms either handcuffs or a theta. 
By Lemma \ref{Hank}, no element of $X$ is in $\cl_M(Y)$, so no edge in $X$ has both ends in $V_G(Y)$.  
By Lemma \ref{Hank}, each element of $X$ is in $\cl_N(Y)$. 
In particular, $b$ and $c$ are in the span of $Y$ in $M/a = N/a$. 
Thus in $G/a$, both ends of $b$ and $c$ are in $V_{G/a}(Y)$. 
This is only possible if $a$, $b$, and $c$ have a common end in $G$. 
But this is contrary to the assumption that $|V_G(X)| \geq 5$. 
\end{proof} 

\begin{claimnumbered} \label{charlap} 
A minimal disagreement set is either a cycle or a pair of vertex disjoint cycles in $G$. 
In the first case each of $a, b, c$ is a chord of the cycle, while in the second case each of $a, b, c$ has one end in each cycle. 
\end{claimnumbered}

\begin{proof}[Proof of Claim]
Let $Y$ be a minimal disagreement set. 
By Lemma \ref{Hank}, each of $a, b, c$ are in the span of $Y$ in $M$, and while $Y$ is independent in $M$, $Y$ is a circuit in $M/z$ for each $z \in X$. 
Thus $Y \cup \{z\}$ is a circuit of $M$, for each $z \in X$.  
Suppose $G[Y]$ contains a vertex $u$ of degree 1. 
Because $|V_G(X)| \geq 5$, there is element $z \in X$ that is not incident with $u$. 
But then the subgraph of $G/z$ induced by $Y$ still has $u$ as a vertex of degree 1, and so is not a circuit of $M/z$, a contradiction. 
Thus $G[Y]$ has minimum degree 2. 
But $Y$ is independent in $M$, so $G[Y]$ forms a collection of disjoint cycles. 
But for each $z \in X$, $Y$ is a circuit of $M/z$, and so induces a bicycle in $G/z$. 
Thus either $G[Y]$ is a cycle and each of $a$, $b$, and $c$ is a chord of $G[Y]$, or $G[Y]$ is a pair of disjoint cycles and each of $a$, $b$, and $c$ have one end in each cycle. 
\end{proof}

Put $\Bb = \{ C : C$ is a circuit of $N$ and a cycle of $G\}$. 
Claims \ref{click} and \ref{charlap} imply that $(G,\Bb)$ is a biased graph. 
Let $\chi$ be the bracelet function mapping a bracelet $Y$ of $(G,\Bb)$ to $\mathsf{dependent}$ if and only if $Y$ is a circuit of $N$. 
Claims \ref{click} and \ref{charlap} along with Theorems  \ref{BraceletFunctionPropriety} and \ref{equivalences} imply that 
$N$ is quasi-graphic, represented by $(G,\Bb,\chi)$. 

\begin{claimnumbered} \label{frameifnoPf}
Either $N$ contains  \emph{Pf} as a minor, or $N$ is a frame matroid represented by the biased graph $(G,\Bb)$, where $\Bb = \{ C : C$ is a cycle of $G$ and a circuit of $N\}$. 
In the latter case, no cycle in $\Bb$ contains an edge in $X$, but each edge in $X$ is a chord of every cycle in $\Bb$.  
\end{claimnumbered}

\begin{proof}[Proof of Claim]
If every minimal disagreement set for $N$ and $M$ induces a cycle in $G$, then $\chi$ maps every bracelet to $\mathsf{independent}$, so $M(G,\Bb,\chi) = F(G,\Bb)$; that is, $N$ is a frame matroid represented by the biased graph $(G,\Bb)$.  
Otherwise, there is a bracelet $Y$ mapped to $\mathsf{dependent}$ by $\chi$. 
Let $C_1, C_2$ be the pair of disjoint cycles whose union is $G[Y]$. 
Each edge $z \in X$ has one end in $V(C_1)$ and one end in $V(C_2)$, and $|V_G(X)| \geq 5$. 
Since no minimal disagreement set contains $a$, $b$, or $c$, the subgraph $H$ induced by $X \cup Y$ contains no cycle that is a circuit of $N$. 
Thus the biased subgraph of $(G,\Bb)$ induced by $X \cup Y$ contains no balanced cycle. 
Writing $H = G[X \cup Y]$, we find as a minor of $N$ the matroid with quasi-graphic representation $(H,\emptyset,\chi')$, where $\chi'$ is the bracelet function assigning $\mathsf{dependent}$ to the bracelet $Y$ in $(H,\emptyset)$ 
(Figure \ref{Pfasaminor}). 
\begin{figure}[tbp] 
\begin{center} 
\includegraphics[scale=0.95]{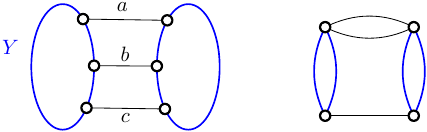}
\end{center} 
\caption{Finding \emph{Pf} as a minor.}
\label{Pfasaminor} 
\end{figure} 
Now it is easy to see that $(H,\emptyset,\chi')$ contains a quasi-graphic representation of \emph{Pf} as a minor. 
So $N$ contains \emph{Pf} as a minor. 
\end{proof}

We may now assume that $N$ is a frame matroid, represented by the biased graph $(G,\Bb)$ as described in Claim \ref{frameifnoPf}. 
Choose a cycle $C$ in $\mathcal{B}$. 
By Claim \ref{frameifnoPf}, none of $a$, $b$, nor $c$ are contained in $C$ but each of $a$, $b$, and $c$ is a chord of $C$. 

If two of the elements in $X$, say $a$ and $b$, are non-adjacent (non-loop) edges whose ends alternate in the cyclic order on $V(C)$, then $(G,\Bb)$ contains a biased graph representing either $\overline{P}$ or \emph{Pf} as a minor (see Figures \ref{LittleGuys2} and \ref{LittleGuys3}). 
Otherwise, if $a, b, c$ form a matching in $G$ then $(G,\Bb)$ contains a biased graph representing either $\overline{\underline{P}}$ or $K_4$ as a minor. 
So assume no two of $a, b, c$ have their ends alternating in the cycle order on $V(C)$ and 
two of $a, b, c$ are adjacent. 
We find as a minor of $(G,\Bb)$ either a biased graph representing M2077 or the biased graph show at left in Figure \ref{figtired}. 
In the later case, since $N/z$ is vertically 3-connected for each $z \in X$, the graph $G/z$ is 2-connected for each $z \in X$. 
It follows that $G$ contains a path $P$ linking a pair of vertices of $C$ as shown in one of the graphs at centre or right in Figure \ref{figtired}. 
\begin{figure}[tbp] 
\begin{center} 
\includegraphics[scale=0.95]{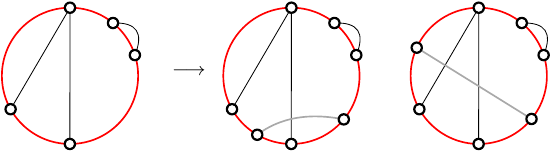}
\end{center} 
\caption{}
\label{figtired} 
\end{figure} 
By Claim \ref{frameifnoPf} every cycle of $C \cup X \cup P$ that contains an edge in $X$ is unbalanced, 
and every cycle in $C \cup P$ that does not have an edge in $X$ as a chord is unbalanced. 
Thus the biased subgraph of $(G,\Bb)$ induced by $C \cup X \cup P$ contains a biased graph representing $\overline{P}$ as a minor. 
\end{proof}

We can now prove the main result of this section. 

\begin{thm} \label{Rachel} 
\addcontentsline{toc}{subsubsection}{No excluded minor has rank more than 7}
Let $N$ be an excluded minor for the class of bicircular matroids. 
Then $N$ has rank less than 
{ten}. 
\end{thm}

\begin{proof} 
Suppose to the contrary that $N$ has rank greater than 
{nine}. 
Let $e \in E(N)$. 
By Lemma \ref{Nvert3}, $N$ is vertically 3-connected, so either $N \del e$ or $N/e$ is essentially 3-connected, by \hyperlink{bixby}{Bixby's Lemma}. 
Thus by Lemma \ref{zevon}, if $N \del e$ is not essentially 3-connected, then $N$ has an element $e'$ for which $N \del e'$ is essentially 3-connected. 
Moreover, Lemma \ref{zevon} guarantees that if $N \del e$ is type 3, then there is an element $e'$ for which $N \del e'$ is essentially 3-connected and type 3. 

So suppose first that $N$ has an element $e$ for which $N \del e$ is essentially 3-connected and type 3. 
By Lemma \ref{eggbert} we may further assume that $N \del e$ is represented by a 2-connected type 3 graph. 
Thus by Lemma \ref{billbill}, there is a bicircular twin $M$ for $N$ relative to a set $X$ of size three, 
satisfying the conditions of Theorem \ref{nomore}. 
Thus $N$ contain one of $M(2C_3)$, $M(K_4)$, $\overline{P}$, \emph{Pf}, $\overline{\underline{P}}$, or M2077 as a proper minor, a contradiction. 

So it must be the case that for every element $e \in E(N)$ for which $N \del e$ is essentially 3-connected, $N \del e$ is type 1 or type 2. 
Moreover, by Lemma \ref{zevon}, $N$ does not have an element whose contraction is essentially 3-connected and type 3. 
By \hyperlink{bixby}{Bixby's Lemma} and Lemma \ref{zevon},  $N$ has an element $e$ for which $N \del e$ is essentially 3-connected. 
Thus by Lemma \ref{guitar}, there is a bicircular twin $M$ for $N$ relative to a set $X$ of size three satisfying the conditions of Theorem \ref{nomore}. 
So again, we find $N$ contains one of $M(2C_3)$, $M(K_4)$, $\overline{P}$, \emph{Pf}, $\overline{\underline{P}}$, or M2077 as a proper minor, a contradiction. 
\end{proof}

\section{There are only a finite number of low-rank excluded minors} \label{greenhornet} 

In this section we prove the following. 

\begin{thm} \label{purple} 
There are only a finite number of excluded minors for the class of bicircular matroids that have rank less than 
{ten}. 
\end{thm}

A pair of elements $e, f$ of a matroid $M$ are \emph{clones} if the bijection from $E(M)$ to $E(M)$ that exchanges $e$ and $f$ while leaving every other element fixed is an automorphism of $M$. 
We also say that the set $\{e,f\}$ is a \emph{clonal pair}. 
The `clone' relation is clearly an equivalence relation on the ground set of a matroid; a \emph{clonal class} of a matroid is an equivalence class under this relation. 
A series pair of elements are clearly clones; 
the following characterisation follows immediately from the definition. 

\begin{prop} \label{clonalpair}
A pair of elements $\{e,e'\}$ are clones if and only if for every circuit $C$ with $|C \cap \{e,e'\}| = 1$, the symmetric difference $C \sd \{e,e'\}$ is a circuit.
\end{prop}

A flat of a matroid is \emph{cyclic} if it is a union of circuits. 
By Proposition \ref{clonalpair}, $e$ and $e'$ are {clones} if and only if every cyclic flat of $M$ that contains one of $e$ or $e'$ contains both $e$ and $e'$. 

A \emph{line} in a matroid is a rank-2 flat. 
A line is \emph{non-trivial} if it contains at least three rank-1 flats. 
Let $M$ be a bicircular matroid represented by the graph $G$, and let $L$ be a non-trivial line of $M$. 
Then the subgraph $G[L]$ induced by $L$ consists of a pair of vertices $u,v$ together with all edges both of whose ends are in $\{u,v\}$. 
We say a line is \emph{adjacent} to another line if their union has rank 3. 
If $L$ and $L'$ are a pair of adjacent lines of $M$, then the set of elements in $L \cap L'$ is precisely the set of loops incident to the vertex $x$, where $V_G(L) \cap V_G(L') = \{x\}$. 
Given a non-trivial line $L$ in a bicircular matroid represented by the graph $G$, let us call the unique pair of vertices $u,v$ in the subgraph of $G$ induced by $L$ the \emph{ends} of the line $L$ \emph{in} $G$. 

Call an element of a non-trivial line \emph{lonely} if it is not also contained in an adjacent line nor properly contained in a rank-1 flat. 
Observe that if $M$ is a bicircular matroid, represented by the graph $G$, then a set of at least three edges sharing a common pair of distinct ends are contained in a non-trivial line and each edge in this set is lonely. 

\begin{prop} \label{green} 
Let $L$ be a non-trivial line in a matroid $M$ and let $e$ and $e'$ be two lonely elements in $L$. 
Then $e$ and $e'$ are clones. 
\end{prop}

\begin{proof} 
Every cyclic flat of $M$ that contains one of $e$ or $e'$ contains both $e$ and $e'$. 
\end{proof} 

\begin{prop} \label{blue} 
Let $L$ be a non-trivial line in a matroid $M$ and let $e, e' \in L$ be lonely elements of $L$. 
If $M \del e$ is bicircular, represented by the graph $G$, then $M \del e'$ is bicircular and is represented by the graph $G'$ obtained by relabelling $e'$ in $G$ by $e$. 
\end{prop}

\begin{proof} 
By Proposition \ref{green}, $e$ and $e'$ are clones. 
\end{proof} 

\begin{lem} \label{brown} 
Let $N$ be an excluded minor of rank at least three. 
Then $N$ does not have as a restriction a line containing more than three lonely elements. 
\end{lem}

\begin{proof} 
Suppose to the contrary that $N$ has a line $L$ containing at least four lonely elements $a,b,c,d$. 
Let $G_a$ be a graph representing $N \del a$. 
Since $L-a$ is a non-trivial line of $N \del a$, $G_a[L-a]$ consists of a pair of vertices $u,v$ together with all edges both of whose ends are in $\{u,v\}$. 
Since none of $b,c,d$ are contained in an adjacent line of $N \del a$, none of $b,c,d$ are loops unless $E(L-a)$ is a pendant set of edges, in which case at most one of these edges may be a loop incident to the pendant vertex; if this is the case then apply a rolling operation so that all edges in $E(L-a)$ are $u$-$v$ edges. 
Let $G_b$ be the graph obtained from $G_a$ by relabelling edge $b$ by $a$, and 
let $G_c$ be the graph obtained from $G_a$ by relabelling edge $c$ by $a$. 
By Proposition \ref{blue}, $G_b$ represents $N \del b$ and $G_c$ represents $N \del c$. 

Let $G$ be the graph obtained from $G_a$ by adding $a$ as an edge linking the ends of $L$, in parallel with edges $b$, $c$, and $d$. 
Consider the circuits of $N$ and $B(G)$: 
\begin{itemize} 
\item $\{a,b,c\}$ is a circuit of both $N$ and $B(G)$;  
\item each circuit of $N$ that contains at most two of $a$, $b$, or $c$, is a circuit of one of $B(G_a)$, $B(G_b)$, or $B(G_c)$; since each of $G_a$, $G_b$, and $G_c$ is equal to the restriction of $G$ to their respective ground sets, each such circuit is a circuit-subgraph of $G$; and, 
\item conversely, every circuit-subgraph of $G$ containing at most two of $a$, $b$, or $c$ is a circuit-subgraph of one of $G_a$, $G_b$, or $G_c$, and so is a circuit of $N$. 
\end{itemize} 
Thus the circuits of $N$ and $B(G)$ coincide, so $G$ represents $N$, a contradiction. 
\end{proof} 

\begin{lem} \label{yellow} 
\addcontentsline{toc}{subsubsection}{$N$ has at most two pairs of parallel elements}
Let $N$ be an excluded minor with rank at least three. 
Then $N$ contains at most three pairs of elements in parallel. 
\end{lem}

\begin{proof} 
Suppose to the contrary that $N$ contain four pairs of elements in parallel, say $\{e,e'\}$, $\{f,f'\}$, $\{g,g'\}$, and $\{h,h'\}$. 
Let $G_e$ be a graph representing $N \del e$. 
Because $\{f,f'\}$, $\{g,g'\}$, and $\{h,h'\}$ remain parallel pairs in $N \del e$, each pair is represented by a pair of loops incident to a vertex of $G_e$; because the three pairs are distinct parallel classes, the pairs are incident to distinct vertices. 
Thus $G_e$ is a type 3 representation for $N \del e$. 
The edge $e'$ is not a loop in $G_e$, because if so, say incident to vertex $v$, then adding $e$ as a loop incident to $v$ would yield a bicircular graph representation for $N$. 
Since $N \del e$ is vertically 3-connected, $G_e$ has no lines or balloons, and since $G_e$ has loops, $G_e$ is not in $\mathscr G$. 
Thus by Theorem 
\ref{goose}, $G_e$ is the unique bicircular representation for $N \del e$. 

Let $G_f$ be a graph representing $N \del f$. 
Applying the argument of the previous paragraph, $G_f$ is the unique bicircular representation for $N \del f$, and each pair $\{e,e'\}$, $\{g,g'\}$, and $\{h,h'\}$ is represented as a pair of loops, no pair incident to the same vertex as another, while $f'$ is not a loop. 

Now consider $N \del \{e,f\}$. 
Because $e$ and $f$ are each contained in a non-trivial parallel class of $N$, $N \del \{e,f\}$ remains vertically 3-connected. 
Moreover, each of $G_e \del f$ and $G_f \del e$ contain loops incident to at least three different vertices (namely, $g$ and $g'$, $h$ and $h'$, along with either $f'$ or $e'$), so each is a type 3 representation for $N \del \{e,f\}$. 
Since $N \del \{e,f\}$ remains vertically 3-connected, neither graph contain a line or a balloon. 
Thus, since neither $G_e \del f$ nor $G_f \del e$ are in $\mathscr G$, by Theorem 
\ref{goose}, 
$G_e \del f = G_f \del e$. 
But this is impossible, because $e'$ is not a loop in $G_e \del f$ while $e'$ is a loop in $G_f \del e$. 
\end{proof} 

\begin{lem} \label{pen} 
Let $N$ be a rank-$r$ excluded minor for the class of bicircular matroids. 
Then $|E(N)| \leq 3 {r \choose 2} + r + 4$. 
\end{lem}

\begin{proof} 
Arbitrarily choose an element $e \in E(N)$, and let $G$ be a graph representing $N \del e$.  
Then $|V(G)| = r$.  
By Lemma \ref{brown}, $N$ does not have a line containing more than three lonely elements, so neither does $N \del e$. 
Thus $G$ has no more than three edges linking each pair of its vertices. 
By Lemma \ref{Nvert3} every parallel class of $N$ has at most two elements, no vertex of $G$ has more than two incident loops. 
By Lemma \ref{yellow}, $N$ has no more than three pairs of elements in parallel, so $N \del e$ has no more than three pairs of elements in parallel. 
Thus $G$ has at most three vertices that have two incident loops. 
Thus $|E(G)| \leq 3 {r \choose 2} + r + 3$, and so $|E(N)| \leq 3 {r \choose 2} + r + 4$.  
\end{proof} 

Theorem \ref{purple} now follows easily. 

\begin{proof}[Proof of Theorem \ref{purple}] 
All non-empty rank-0 and all rank-1 matroids are bicircular. 
A rank-2 matroid is bicircular if and only if it has at most two rank-1 flats of size greater than one. 
It is easy to see that a rank-2 matroid has more than two rank-1 flats containing at least two elements if and only if it contains $M(2C_3)$ as a minor. 
So $M(2C_3)$ is the only excluded minor of rank 2. 
For 
{$r \in \{3, 4, \ldots, 9\}$}, apply Lemma \ref{pen}: 
an excluded minor of rank $r$ for the class of bicircular matroids has at most 
$3 {r \choose 2} + r + 4$ elements. 
\end{proof} 


\appendix 
\section*{Appendix}
\section{Excluded minors of small rank}\label{appendix}

Mayhew and Royle have catalogued all matroids on up to nine elements. 
Their computation is described in~\cite{MR2389607}; the dataset is available at \url{https://doi.org/10.5281/zenodo.6825419}. 
We identified twenty-seven minor-minimal non-bicircular matroids among those in the dataset, as follows.  
First, we generated a complete list of all connected graphs with up to nine edges (loops and multiple edges permitted). 
Next we identified the bicircular matroid represented by each of these graphs in the dataset, and declared a single matroid loop to be bicircular. 
We then declared a matroid in the dataset to be bicircular exactly when each of its components had been identified as bicircular. 
Excluded minors in the dataset were then identified as just those non-bicircular matroids all of whose single-element deletions and single-element contractions are bicircular. 

The table on the following page exhibits the twenty-seven excluded minors so discovered. 
Points in geometric representations are shown as solid discs; a pair of points in parallel is indicated by a circled solid disc. 
The M-numbers labelling the matroids are their unique identifiers in the dataset. 


\begin{figure}[htbp]
\begin{center}
\includegraphics[scale=1.2,trim={0 6.5cm 8cm 8cm},clip]{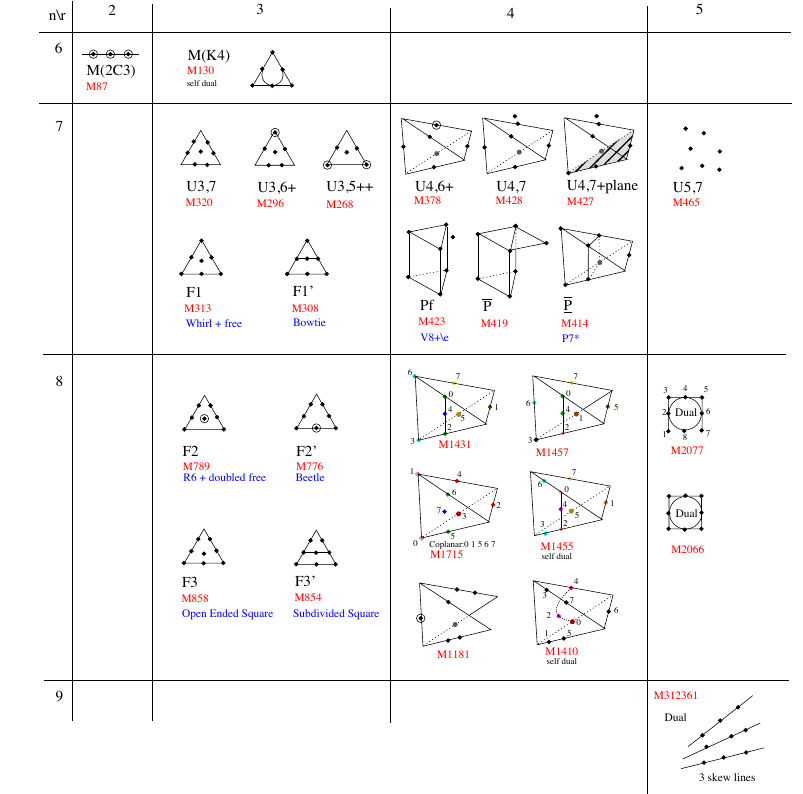}
\caption{Excluded minors for the class of bicircular matroids with at most nine elements.}
\end{center}
\label{bungalowbill}
\end{figure}

\newpage 

\section*{Acknowledgments} 
Our sincere thanks to the referees, whose comments and suggestions greatly improved the paper. 

\bibliographystyle{amsplain}


\begin{aicauthors}
\begin{authorinfo}[matt]
  Matt DeVos\\
  Department of Mathematics, Simon Fraser University\\
  Burnaby, British Columbia, Canada\\
  mdevos\imageat{}sfu\imagedot{}ca 
\end{authorinfo}
\begin{authorinfo}[daryl]
  Daryl Funk\\
  Department of Mathematics, Douglas College\\
  New Westminster, British Columbia, Canada\\
  funkd\imageat{}douglascollege\imagedot{}ca 
\end{authorinfo}
\begin{authorinfo}[luis]
  Luis Goddyn\\
  Department of Mathematics, Simon Fraser University\\
  Burnaby, British Columbia, Canada\\
  goddyn\imageat{}sfu\imagedot{}ca
\end{authorinfo}
\begin{authorinfo}[gordon]
  Gordon Royle\\
  School of Mathematics and Statistics, University of Western Australia\\
  Perth, Australia\\
    gordon\imagedot{}royle\imageat{}uwa\imagedot{}edu\imagedot{}au 
\end{authorinfo}
\end{aicauthors}

\end{document}